\documentclass{amsart}
\usepackage[utf8]{inputenc}
\usepackage{amssymb}
\usepackage{amsthm}
\usepackage{amsmath}
\usepackage{enumerate}
\usepackage{mathtools}
\usepackage{comment}
\usepackage{todonotes}
\usepackage{xcolor}\mathtoolsset{showonlyrefs}
\usepackage{hyperref}

\usepackage{lineno,hyperref}
\usepackage{amsthm}
\usepackage{amsmath}
\usepackage{amsfonts}
\usepackage{amssymb}

\usepackage{enumitem}
\newtheorem{thm}{Theorem}[section]
 \newtheorem{cor}[thm]{Corollary}
 \newtheorem{lem}[thm]{Lemma}
 \newtheorem{prop}[thm]{Proposition}
 \theoremstyle{definition}
 \newtheorem{defn}[thm]{Definition}
 \theoremstyle{remark}
 \newtheorem{rem}[thm]{Remark}
 
 \numberwithin{equation}{section}
 \numberwithin{thm}{section}

\DeclareMathOperator*{\esssup}{ess\,sup}

\title[Decomposable operators acting between distinct $L^p$-direct integrals]{Decomposable operators acting between distinct $L^p$-direct integrals of Banach spaces 
}

\author{Nikita Evseev}
	\address{Institute for Advanced Study in Mathematics, Harbin Institute of Technology, 150006 Harbin, China.
              Sobolev Institute of Mathematics, 4 Academic Koptyug avenue,  630090 Novosibirsk, Russia.
	}
	\email{evseev@math.nsc.ru}
\thanks{The first named author was partly supported by Russian Foundation for Basic Research (project no. 20-01-00661).
The second named author was supported by the grant GA\v{C}R 20-19018Y}

\author{Alexander Menovschikov}
	\address{Department of Mathematics, University of Hradec Kr\'alov\'e, Rokitansk\'eho 62, 500 03 Hradec Kr\'alov\'e, Czech Republic.
              Faculty of Economics, University of South Bohemia, Studentsk\'a 13, 370 05 \v{C}esk\'e Bud\v{e}jovice, Czech Republic.
			  	}
	\email{alexander.menovschikov@uhk.cz}

\begin{document}

\maketitle

\begin{abstract}
The notion of decomposable operators acting between distinct $L^p$-direct integrals of Banach spaces is introduced. 
We show that these operators generalize the composition operator, in sense that a mapping is replaced by a binary relation.
The necessary and sufficient conditions for the boundedness of those operators are the main results of the paper.\\
Keywords: Direct integral; $L^p$-direct integral; composition operator; decomposable operators; mixed-norm Lebesgue space
\end{abstract}

\section{Introduction}

In this article we study a particular class of bounded operators on $L^p$-direct integrals of measurable families of Banach spaces. 
More precisely, we extend the well-known concept of decomposable operators, acting from a direct integral into themselves, to the case when such an operator transforms one direct integral into another that is substantially different from the original one. The idea and our initial motivation for such generalization is closely related to the study of the composition operator. In this sense, we brunch out from the classical field of application of decomposable operators and show their connection with the composition and multiplication operators on Banach spaces and, as an application of the results obtained, we prove theorems on the boundedness of the composition operator on mixed-norm Lebesgue and Sobolev spaces.

The concept of direct integral of Banach spaces (as generalization of a direct sum) was first introduced in von Neumann’s papers in the early 1950s and used to classify von Neumann algebras \cite{N49}.
The main purpose was closely connected to the formalization of quantum mechanics. 
Initially, the theory considered the direct integrals of Hilbert spaces (see \cite{W70,Az74}),
which provides finer analytical tools,
for example the direct integral version of the spectral theorem (\cite{G96}, \cite[Theorem 7.1]{H13}).
The direct integrals play an important role in the theory of group representations (\cite{N80,NS82,JR17}). 
In addition, various issues of functional analysis are actively studied  
(e.g. \cite{Schwartz1967,DNSZ2016,Ha17-arxiv,Ab2020}).  
Other fields of applications are processes on porous medium \cite{SW91,MB08,RLMV2021} and evolutionary problems 
\cite{ACDE2021,EM2020jmaa}. In both fields the direct integrals provide a natural mathematical description for varying structures.


\subsection{Definitions and Objectives}
Let $(\Omega, \Sigma_\Omega, \mu)$ be a measure space and $\{B_\omega\}_{\omega\in\Omega}$ be a measurable family of Banach spaces. 
Then the direct integral $\int_{\Omega}^{\oplus}B_\omega\,d\mu$ is a set of all measurable  functions $f:\Omega\to\cup B_\omega$, such that $f(\omega)\in B_\omega$. 
The $L^p$-direct integral $\left(\int_{\Omega}^{\oplus}B_\omega\,d\mu\right)_{L^p}$ is a set of all measurable functions such that $\int_\Omega\|f(\omega)\|^p_\omega\, d\mu$ is finite.


One of the most natural classes of operators on those spaces is the so-called decomposable operators. 
A decomposable operator on a $L^p$-direct integral is a map $\omega\mapsto\mathcal L(B_\omega)$. This means that the action of such an operator can be represented as a pointwise action of the family of linear operators: $P[f](\omega) = P(\omega)[f(\omega)]$ for every section $f \in \int_{\Omega}^{\oplus}B_\omega\,d\mu$.

Our first main results (Theorem \ref{theorem:MultiplicationOperator} and Theorem \ref{theorem:operatorQ}) are aimed to describe bounded decomposable operators of the form
\begin{equation}\label{operatorQ}
    P \colon \left(\int_{\Omega}^{\oplus}B_\omega\,d\mu\right)_{L^p} \to \left(\int_{\Omega}^{\oplus}D_\omega\,d\nu\right)_{L^q},
\end{equation}
where $\{D_\omega\}_{\omega\in\Omega}$ is another measurable family on $(\Omega, \Sigma_\Omega, \nu)$.

Developing this idea, we come to the main object of this article. We give a meaningful description of decomposable operators acting between different $L^p$-direct integrals:
\begin{equation}\label{operatorM}
  M : \left(\int_{T}^{\oplus}W_t\,d\mu\right)_{L^p} \to \left(\int_{S}^{\oplus}V_s\,d\lambda\right)_{L^q}.    
\end{equation}
To study operators \eqref{operatorM} properly, 
we have to introduce an auxiliary construction, which is based on the following observation.     
Consider a composition operator $C_\varphi:L^p(\Omega', \mu) \to L^q(\Omega, \nu)$ which is induced by a mapping $\varphi:\Omega\to\Omega'$ 
and defined by the rule $C_\varphi f = f\circ\varphi$. 
This operator can be split into two operators: the first acts onto functions defined on graph $\Gamma_\varphi$, and the second is an isomorphism, induced by projection from the graph to $\Omega$.    
\begin{figure}[h]
\centering	
\includegraphics[width=0.45\linewidth]{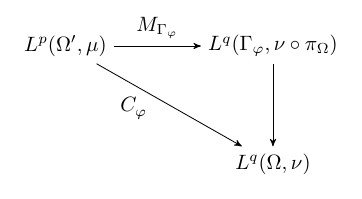}
\caption{Splitting the composition operator}\label{fig:fig1}
\end{figure}
It turned out that the boundedness of the composition operator $C_\varphi$ is equivalent to the boundedness of operator $M_{\Gamma_\varphi}$  (see Fig. \ref{fig:fig1}).
This connection is one of the main motivations of our research. It gives us a new view for studying composition operators on function spaces and allows us to consider a binary correspondence instead of mapping. Moreover, the construction of an operator acting on a direct integral defined over a Cartesian product (now, instead of $\Gamma_\varphi$ we consider an arbitrary set $F \subset S \times T$) is interesting from the point of view of the theory of decomposable operators in its classical sense.
 
Finally, we apply the developed methods for describing composition operator on   
mixed-norm Lebesgue and Sobolev spaces
\begin{equation}
C_\varphi:L^{p,\beta}(U') \to L^{q,\alpha}(U)    
\end{equation}
and 
\begin{equation}
C_{\varphi}:L^p(T, W^{1,\beta}(U'_t)) \to L^q(S, W^{1,\alpha}(U_s)).
\end{equation}

\subsection{Structure of the paper}
In Section \ref{section:Preliminaries}, we give a short introduction to the theory of the direct integral, as well as the definitions of the main objects of our research (the concept of a lattice supremum, which is used to obtain a measurable analog of the norm of an operator; K\"othe spaces, which provide a more general scalar characteristic for the spaces under study than $L^p$-spaces; spaces of a homogeneous type, which are both a suitable basis for constructing a direct integral, and spaces that still have enough material for the use of analytical tools).
In Section \ref{section:Decomposable}, we pass from classical definition of decomposable operators to a more general description by using the notion of measurable family of bounded linear operators.
Further, in Section \ref{section:MFoperators}, based on the observations from the previous section, the main objects is defined. 
We establish the conditions for the boundedness of decomposable operators acting between different direct integrals.  
In Section \ref{section:WeightedCO}, we show a relationship with the classical composition operator and implement the scheme that was presented on Fig. \ref{fig:fig1}. 
In conclusion, in Section \ref{section:mixed}, we give two examples that, on the one hand, demonstrate the action of such operators, and, on the other hand, 
establish their relation with the porous medium models. 





\section{Preliminaries}\label{section:Preliminaries}

\subsection{K\"othe function Spaces}

As a scalar norm on a direct integral, the norm in the K\"othe space is traditionally considered. It is suitable for describing general questions, but when we turn to describing decomposable operators acting between different direct integrals, we are forced to limit ourselves to using a more specific norm, namely the $L^p$-norm. 
Therefore, all the results below Section \ref{section:Decomposable} are formulated for $L^p$-spaces.

\begin{defn}
Given a measure space $(\Omega, \Sigma_\Omega, \mu)$. A \textit{K\"othe function space} is a Banach space $(K, \|\cdot\|_K)$ of real-valued measurable functions on $\Omega$ modulo equality almost everywhere 
such that  
\begin{enumerate}[label=(\roman*)]
\item $\chi_A\in K(\Omega)$ for every measurable $A$ with $|A|<\infty$;
\item every $f\in K(\Omega)$ is integrable over measurable set $A$ whenever $|A|<\infty$; 
\item if $g$ is measurable and $f\in K(\Omega)$ such that $|g(x)|\leq |f(x)|$ a.e. then $g\in K(\Omega)$ and $\|g\|_{K(\Omega)}\leq\|f\|_{K(\Omega)}$.
\end{enumerate}
\end{defn}

Any Banach function space is an example of the above notion, 
in particular, 
the Lebesgue space $L^p(\Omega)$, $1\leq p\leq \infty$. 
As a reference we suggest the book by P.K. Lin \cite{Lin2004}.

We will need the following property.
\begin{prop}[{\cite[Lemma IV.3.2, p. 97]{Kantorovich1982}}]\label{prop:strong-ae}
If a sequence $\{f_n\}$ converges to $f$ in $K(\Omega)$, then there exists a subsequence $\{f_{n_k}\}$ 
that converges to $f$ a.e.
\end{prop}

A K\"othe function space $K(\Omega)$ has \textit{the Fatou property} if for any 
increasing nonnegative sequence $\{f_n\}$ bounded in norm $\|\cdot\|_K$,
its limit $f(\omega)=\lim\limits_{n\to\infty}f_n(\omega)$ belongs to $K(\Omega)$
and $\|f\|_K = \lim\limits_{n\to\infty}\|f_n\|_K$.

\subsection{$L^p$-direct integrals of Banach spaces}
Here we provide a brief account of the direct integral theory (also see \cite{N80,HLR91,JR17}). 
Let $(\Omega, \Sigma_\Omega, \mu)$ be a measure space. 
We will omit the symbol of $\sigma$-algebra and will use a notation $(\Omega, \mu)$, so that it cannot cause ambiguity. Let $V$ be a vector space.
A family of semi-norms $\{\|\cdot\|_\omega\}_{\omega \in \Omega}$ on $V$ is said to be 
a \textit{measurable family of semi-norms}, if the function $\omega \mapsto \|v\|_\omega$ is measurable for each $v\in V$. 
We define Banach spaces $B_\omega$ to be a completion $V/\ker(\|\cdot\|_\omega)$ with respect to semi-norms $\|\cdot\|_\omega$.
Such a family  $\{B_\omega\}_{\omega \in \Omega}$ is called a \textit{measurable family of Banach spaces} over $(\Omega, \mu, V)$.

A mapping $f:\Omega\to\bigcup_{\omega\in \Omega}B_\omega$ is said to be a \textit{section} of family $\{B_\omega\}_{\omega\in \Omega}$, if $f(\omega)\in B_\omega$ for all $\omega \in \Omega$.
A \textit{simple section} is a section $h$ for which there exist $n\in\mathbb N$, $v_1, \dots, v_n\in V$ and measurable sets $A_1, \dots, A_n\subset \Omega$ such that 
$f(\omega) = \sum_{k=1}^n\chi_{A_k}\cdot v_k$ for all $\omega\in \Omega$. Further, we shall also call sections just functions 
(see Fig \ref{fig:1}).

\begin{figure}[h]
\centering	
\includegraphics[width=0.8\linewidth]{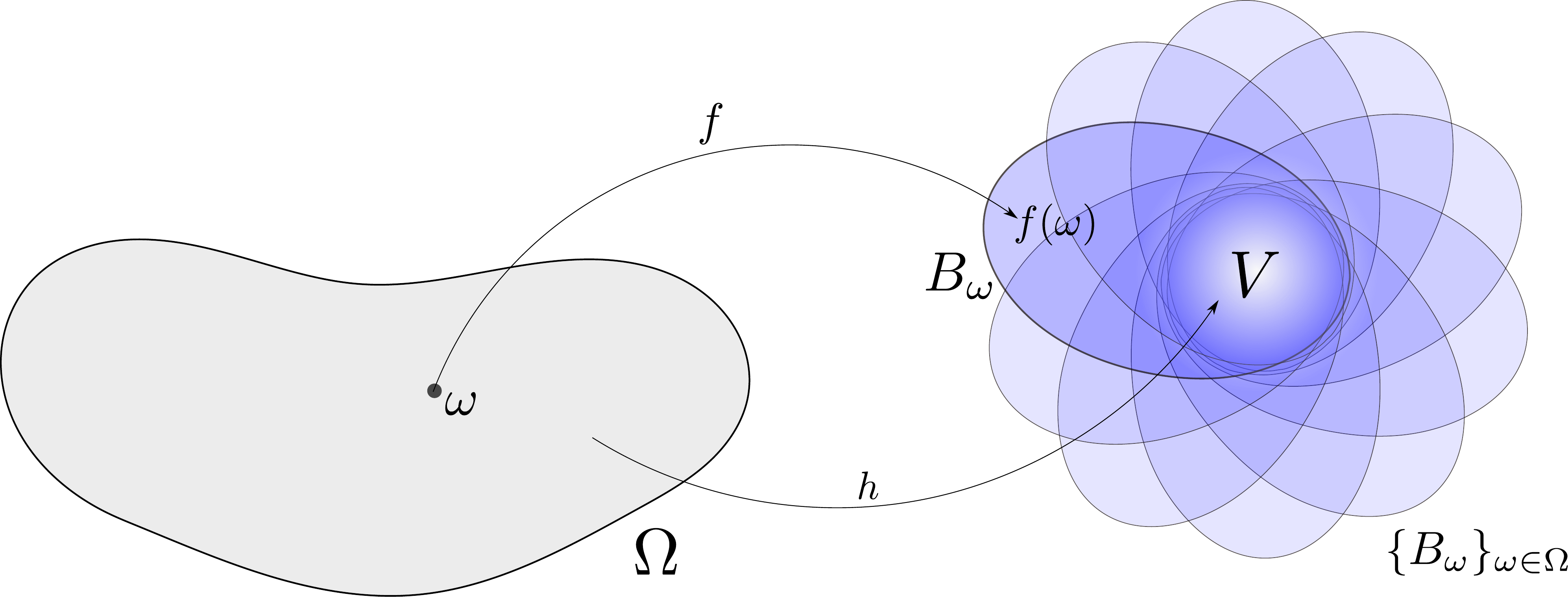}
\caption{A section $f$ takes values in a Banach space $B_\omega$ depending on $\omega$, whereas a simple section $h$ takes values in the vector space $V$, which is in the intersection of all $B_\omega$.} \label{fig:1}
\end{figure}

\begin{defn}
A section $f$ of $\{B_\omega\}_{\omega\in \Omega}$ is said to be \textit{measurable} if there exists a sequence of simple sections $\{h_k\}_{k\in\mathbb N}$ such that, for a.e. $\omega\in \Omega$, $h_k(\omega)\to f(\omega)$ in $B_\omega$ as $k\to\infty$.
\end{defn}

Let us define the \textit{direct integral} of a measurable family of Banach spaces $\{B_\omega\}_{\omega\in \Omega}$ with respect to measure $\mu$ as a space of all equivalence classes of measurable sections, and denote it $\int_{\Omega}^{\oplus}B_\omega\,d\mu$.
Note that for a measurable section $f$ the function $\omega \mapsto \|f(\omega)\|_{\omega}$ is measurable in the usual sense.

Let $K(\Omega)$ be a K\"othe function spaces.
We define a \textit{$K$-direct integral} 
$\left(\int_{\Omega}^{\oplus}B_\omega\,d\mu\right)_{K}$
as a space of all $f\in \int_{\Omega}^{\oplus}B_\omega\,d\mu$ such that the function $\|f(\omega)\|_{\omega}$ belongs to $K(\Omega)$.
Endowed with the norm
$$
\|f\|_{K(\Omega,\{B_\omega\})} 
:=  \big\|\|f(\cdot)\|_{\cdot}\big\|_K
$$
it becomes a Banach space.

Then, for every $p\in[1,\infty]$ a \textit{$L^p$-direct integral} 
$\left(\int_{\Omega}^{\oplus}B_\omega\,d\mu\right)_{L^p}$ is a Banach space with the norm
$$
\|f\|_{L^p(\Omega,\{B_\omega\})} 
:=  \left(\int_{\Omega}\|f(\omega)\|^p_{\omega}\,d\mu\right)^{\frac{1}{p}}.
$$

An important example of $L^p$-direct integral is a Lebesgue space with mixed norm.
Indeed, let $(S, \Sigma_S, \nu)$, $(X, \Sigma_X \eta)$ be $\sigma$-finite measure spaces,
$U\subset S\times X$, and $U_s = \{x\in X | (s, x)\in U\}$. 
Define a mixed-norm Lebesgue space $L^{q,\alpha}(U)$ as a set of measurable functions $f(s,x)$ with
finite norm
$$ 
||f||_{L^{q,\alpha}(U)} = \bigg(\int_{S} \bigg(\int_{U_s} |f(s,x)|^\alpha \, d\eta(x)\bigg)^{\frac{q}{\alpha}} \, d\nu(s)\bigg)^{\frac{1}{q}}.
$$
The next lemma shows that elements of the mixed-norm Lebesgue space are measurable sections with respect to the given family. 
\begin{lem}\label{mixlem} 
Let $(S,\Sigma_S, \nu)$, $(X,\Sigma_X, \eta)$ be $\sigma$-finite measure spaces,
and let $(R,\Sigma_R, \sigma)$ be their product. 

Let $1\leq p<\infty$, $U\in\Sigma_{R}$, and let $f$ be $\sigma$-measurable function on $(U, \Sigma_R\cap U, \sigma)$
such that $f(s,\cdot)$ is in $L^p(U_s, \Sigma_X\cap U_s, \nu)$ for $\nu$-a.e. $s\in S$.

Then, $f$ is a measurable section of $\{L^{p}(U_s, \Sigma_X\cap U_s, \nu)\}_{s\in S}$.
\end{lem}
\begin{proof}
The method of the proof is the very same as in 
\cite[Lemma III.11.16, p.196]{DunfordSchwartz1988}.
\end{proof}
The last lemma allows us to regard mixed-norm Lebesgue space as $L^q$-direct integral: 
$$
L^{q,\alpha}(U) = \left(\int_{S}^{\oplus}L^\alpha(U_s)\,d\nu \right)_{L^q}.
$$

\subsection{Analysis on homogeneous spaces}
Starting from Section \ref{section:MFoperators}, we need to differentiate set functions. 
Thus, we are forced to restrict ourselves to considering spaces of homogeneous type instead of general measure spaces.

\begin{defn}
\textit{A homogeneous space} $(T,d,\mu)$ is a quasimetric space with Borel measure $\mu$, which has the following two properties:
\begin{enumerate}
\item There is a constant $c_1\geq 1$ such that for all $t_1,t_2\in T$ and $r>0$
if balls $B(t_1,r)$, $B(t_2,r)$ intersect, then $B(t_2,r)\subset B(t_1,c_1r)$ (\textit{absorption});
\item For any ball $B(t, r)\subset T$ of positive radius and $0<\mu(B(T, r))<\infty$
 there is a constant $c_2$, that $\mu(B(t, {c_1r})) \leq c_2\mu(B(t,r))$ (\textit{doubling}).
 \end{enumerate}
\end{defn}

We will assume that homogeneous spaces are $\sigma$-finite measure spaces.

One of the main tools we apply in this article is the Radon--Nikod\'ym theorem 

\begin{thm}[{\cite[Theorem B, p. 128]{Halmos1974}}]
If $(T, \Sigma_T, \mu)$ is a $\sigma$-finite measure space and if a $\sigma$-finite signed measure $\nu$ on $\Sigma_T$ is absolutely continuous with respect to $\mu$, then there exists a finite valued measurable function $J$ on $T$ such that $\nu(E) = \int_E J d\mu$ for every measurable set $E$. The function $J$ is unique.
\end{thm}

A direct consequence of this theorem is the following \textit{change of variable formula}:

\begin{thm}[{\cite[Theorem D, p. 164]{Halmos1974}}]
If $\varphi$ is measurable transformation from a measure space $(T, \Sigma_T, \mu)$ into a $\sigma$-finite measure space $(S, \Sigma_S, \nu)$, such that $\mu\circ\varphi^{-1}$ is absolutely continuous with respect to $\nu$, then there exists a nonnegative measurable function $J_{\varphi^{-1}}$ (Radon--Nikod\'ym  derivative) on $S$ such that
\begin{equation}\label{formula_change}
\int_{T}f(\varphi(t))\, d\mu = \int_S f(s)J_{\varphi^{-1}}(s)\,  d\nu,
\end{equation}
for every measurable function $f$.
\end{thm}

\begin{defn}
The mapping $\varphi: T \to S$ enjoys Luzin $\mathcal N^{-1}$-property if the inverse image of the set of measure zero also has measure zero. 
The mapping $\varphi: T \to S$ enjoys Luzin $\mathcal N$-property if the image of the set of measure zero also has measure zero.
\end{defn}

In the mapping theory it is more natural to assume Luzin $\mathcal N^{-1}$-property for $\varphi$ in change of variable formula. This property guarantees that $\mu\circ\varphi^{-1}$ is absolutely continuous with respect to  $\nu$. 
In the case of homogeneous spaces the \textit{Radon--Nikod\'ym  derivative} $J_{\varphi^{-1}}(s)$ can be calculated as the volume derivative of the inverse mapping:
$$
J_{\varphi^{-1}}(s) = \lim\limits_{r\to 0}\frac{\mu(\varphi^{-1}(B(s,r)))}{\nu(B(s,r))} 
\quad \text{for $\nu$-a.e. } s\in S.
$$

In article \cite{VU04}, the following properties of countable additive set functions on homogeneous spaces were obtained:

\begin{prop}[{\cite[Corollary 5.]{VU04}}]\label{prop:setderivative}
Let $\Phi$ be a monotone countable additive set function, which defines on some family of open sets in the homogeneous space $(T,d,\mu)$ and takes finite values on open balls. 
Then

1) there exists a finite derivative for almost all $t\in T$ 
$$
\lim\limits_{r\to 0}\frac{\Phi(B(t, r))}{\mu(B(t, r))} = \Phi'(t); 
$$ 

2) $\Phi':T\to\mathbb R$ is a measurable function;

3) for any open set $U$ belonging to the given family, the next inequality holds
\begin{equation}
\int_U\Phi'(t)\, d\mu \leq \Phi(U).
\end{equation} 
\end{prop}

Also, we will need the following version of \textit{the Lebesgue differentiation theorem} 
(see, e.g. \cite[Corollary 3]{VU04}).
\begin{thm}
Let $(T,d,\mu)$ be a homogeneous space, $U$ be a domain in $T$ and $f\in L^1_{\operatorname{loc}}(U)$, then, for almost all $t_0 \in U$,
$$
\lim\limits_{r\to 0}\frac{1}{\mu(B(t_0, r))}\int_{B(t_0, r)}
|f(t) - f(t_0)| \, d\mu = 0.
$$
\end{thm}

\subsection{Lattice supremum}\label{sec:latsup}
In the course of describing bounded decomposable operators, we need to take the supremum of an uncountable family of measurable functions. 
The pointwise supremum is not necessarily measurable. The reasonable measurable analog is the lattice supremum.

\begin{defn}\label{defn:LatticeSup}
Let $\mathcal{F}$ be a family of measurable functions on $\Omega$.
The lattice supremum $\bigvee\mathcal{F}$ is a function that satisfies
the following two properties:
\begin{enumerate}
\item $f\leq \bigvee\mathcal{F}$ a.e. for any $f\in \mathcal{F}$
\item If $g$ is a measurable function s.t.  $f\leq g$ a.e. for all $f\in \mathcal{F}$, then  $\bigvee\mathcal{F}\leq g$ a.e.
\end{enumerate}
\end{defn}
The measurability of $\bigvee\mathcal{F}$ is the consequence of the next lemma.
\begin{lem}\label{lemma:LatticeSup}
Let $\mathcal{F}$ be a class of measurable functions defined in a measurable set $\Omega$. 
Then $\bigvee\mathcal{F}$ exists and there is a countable subfamily
 $\mathcal{G}\subset\mathcal{F}$ such that 
$$
\bigvee\mathcal{F}=\bigvee \mathcal{G}=\sup \mathcal{G}. 
$$
\end{lem}
\begin{proof}
The complete proof could be found in  \cite[Lemma 2.6]{HajlaszMaly2002} or \cite[Lemma 2.6.1]{M-N1991}. 
Here we provide only a rough idea:
there exists a sequence $\{f_k\}\subset \mathcal{F}$ such that 
$\lim\limits_{k\to\infty}\max\{f_1,\dots,f_k\}$ is the desired function $\bigvee\mathcal{F}$.
\end{proof}

\begin{prop}\label{prop:supBFS}
Let $(K(\Omega),\|\cdot\|_K)$ be a K\"othe function space having the Fatou property.
If $\mathcal F\subset K(\Omega)$ is bounded in norm  $\|\cdot\|_K$
and closed under finite maxima, then
$\bigvee\mathcal{F}\in K(\Omega)$. 
Moreover, $\bigvee\mathcal{F}$ is a pointwise limit of an
increasing sequence of functions from~$\mathcal{F}$.
\end{prop}
\begin{proof}
From the proof of Lemma \ref{lemma:LatticeSup} we know that 
$\bigvee\mathcal{F}$ is a limit of increasing sequence,
which consists of finite maxima. Then, by assumption, this sequence is in $\mathcal{F}$.
Due to the Fatou property, its limit belongs to $K(\Omega)$. 
\end{proof}

\section{Decomposable operators}\label{section:Decomposable}

In the article \cite{JR17}, authors give an exhaustive description of decomposable operators on $L^p$-direct integrals. They consider such operators as a natural generalization of the corresponding notion in the theory of the usual direct integral of separable Hilbert spaces and apply them to the group representations theory. 
We are modifying this approach to be able to extend the notion of decomposable operator to more general settings. 

Let $\{B_\omega\}_{\omega\in\Omega}$ be a measurable family of Banach spaces over 
$(\Omega,\mu, V)$, originating from some measurable family of semi-norms $\{\|\cdot\|_\omega\}_{\omega\in\Omega}$ on vector space $V$.

\begin{defn}\label{def:operators1}
\textit{A measurable family of bounded linear operators}$\{P(\omega)\}_{\omega \in \Omega}$ is a map $\omega\mapsto P(\omega)\in\mathcal{L}(B_{\omega})$, defined for almost all $\omega \in \Omega$, such that
\begin{enumerate}
\item\label{M1} A map $\omega\mapsto P(\omega)v$ is a measurable section of $\{B_{\omega}\}_{\omega\in \Omega}$ for any $v\in V$;
\item\label{M2} $\|P(\omega)\|_{\mathcal B(B_\omega)} <\infty$ almost everywhere on $\Omega$.
\end{enumerate}
\end{defn}

Condition 1. of Definition \ref{def:operators1} can be easily extended from $v \in V$ to all measurable sections. 

\begin{lem}\label{lemma:MeasurableSection}
Let $\{P(\omega)\}_{\omega\in \Omega}$ be a measurable family of bounded linear operators,
then, for any measurable section $\zeta(\omega)$ of $\{B_\omega\}_{\omega\in \Omega}$, $P(\omega)\zeta(\omega)$ is a measurable section of $\{B_{\omega}\}_{\omega\in \Omega}$.
\end{lem}
\begin{proof}
Let $\zeta(\omega)$ be a measurable section, then there is a sequence of simple sections $s_k(\omega)\to\zeta(\omega)$ a.e.
From Definition \ref{def:operators1} we obtain that $P(\omega)s_k(\omega)$ is a measurable section,
and $P(\omega)s_k(\omega) \to P(\omega)\zeta(\omega)$ in $B_\omega$ a.e.
Thus, $P(\omega)\zeta(\omega)$ is a measurable section.
\end{proof}

The set $\mathcal M$ of all measurable families of bounded linear operators as in Definition \ref{def:operators1}
is an $L^{\infty}(\Omega,\mu)$-module (in the sense that $\mathcal M$ is a generalized vector space with functions from $L^{\infty}$ instead of scalars). 
Define a mapping $N:\mathcal M\to L^0(\Omega, \mu)$ which maps a measurable family to a nonnegative measurable function by the rule
$$
N(P)(\omega) := \bigvee \left\{ \frac{\|P(\omega)v\|_\omega}{\|v\|_\omega} : v\in V\right\},
$$
with the convention that $\frac{\|P(\omega)v\|_\omega}{\|v\|_\omega} = 0$ whenever $\|v\|_\omega = 0$.
$\mathcal M$  equipped with the random norm $N$ is a randomly normed space \cite[Chapter 5]{HLR91}.

Let us note that the introduced mapping $N(P)(\omega)$ is a measurable analog of the operator norm $\|P(\omega)\|_{\mathcal B(B_\omega)}$, which is not necessarily measurable. 

\begin{lem}\label{lemma:NsupK}

The value of $N(P)(\omega)$ will not change if the lattice supremum is taken over all measurable sections, i.e.

\begin{equation}\label{eq:Nsup}
N(P)(\omega) = \bigvee \left\{ \frac{\|P(\omega)f(\omega)\|_\omega}{\|f(\omega)\|_\omega} : f \text{ is a measurable section of } \{B_\omega\} \right\}.
\end{equation}
\end{lem}
\begin{proof}
Denote the right hand side of \eqref{eq:Nsup} as $K(\omega)$.
It is clear that for any $v\in V$ 
\begin{equation}\label{eq:NsupK:eq1}
\frac{\|P(\omega)v\|_\omega}{\|v\|_\omega} \leq K(\omega).
\end{equation}
Let $g:\Omega\to\mathbb R$ be a measurable function such that for any $v\in V$
$$
\frac{\|P(\omega)v\|_\omega}{\|v\|_\omega} \leq g(\omega) \quad \text{a.e.} 
$$
For any simple section $s(\omega) =  \sum_{k=1}^n\chi_{A_k}(\omega)\cdot v_k$ we have
$$
\frac{\|P(\omega)s(\omega)\|_\omega}{\|s(\omega)\|_\omega} = \sum_{k=1}^n\frac{\|P(\omega)v_k\|_\omega}{\|v_k\|_\omega} \chi_{A_k}(\omega)
\leq g(\omega) \quad \text{a.e.}
$$
Now if $f(\omega)$ is a measurable section, then there is a sequence of simple sections which converges a.e. on $\Omega$.
Therefore, $\frac{\|P(\omega)s(\omega)\|_\omega}{\|s(\omega)\|_\omega}\leq g(\omega)$ a.e. and 
\begin{equation}\label{eq:NsupK:eq2}
K(\omega) \leq g(\omega).
\end{equation}
Thus, by Definition \ref{defn:LatticeSup}, equations \eqref{eq:NsupK:eq1}, \eqref{eq:NsupK:eq2} imply \eqref{eq:Nsup}.
\end{proof}

Based on Definition \ref{def:operators1}, we introduce the definition of a decomposable operator.

\begin{defn}
A linear operator $P$ in $\int_{\Omega}^{\oplus}B_\omega\,d\mu$
is said to be \textit{decomposable} if it is defined by the measurable family of bounded operators 
$\{P(\omega)\}$, in the sense that
\begin{equation}\label{def:M}
P[f](\omega) = P(\omega)[f(\omega)]
\end{equation}
for every $f\in \int_{\Omega}^{\oplus}B_\omega\,d\mu$.
\end{defn}

Our notion of decomposable is weaker than one from \cite{JR17}. The difference is that we do not demand the essential boundedness of function $\omega \mapsto \|P(\omega)\|$. 
Instead, we derive this property from mesurability of the family $\{P(\omega)\}$:

\begin{thm}\label{theorem:MultiplicationOperator}
Let $\{P(\omega)\}$ be a measurable family of bounded operators.
Then $Pf\in \left(\int_{\Omega}^{\oplus}B_\omega\,d\mu\right)_K$ for every $f\in \left(\int_{\Omega}^{\oplus}B_\omega\,d\mu\right)_K$
if and only if $N(P) \in L^{\infty}(\Omega)$.
In this case the induced operator $P:\left(\int_{\Omega}^{\oplus}B_\omega\,d\mu\right)_K \to \left(\int_{\Omega}^{\oplus}B_\omega\,d\mu\right)_K$
is bounded and $\|P\| = \esssup\limits_{\Omega} N(P)(\omega)$. 
\end{thm}
\begin{proof}
First, we prove that the graph $\Gamma_P$ of operator $P$ is a closed set.
Let $(f_n,g_n)\in \Gamma_P$, where $g_n(\omega) = P(\omega)f_n(\omega)$.
Suppose $(f_n,g_n)$ converges to $(f,g)$, i.e. $f_n\to f$ in $\left(\int_{\Omega}^{\oplus}B_\omega\,d\mu\right)_K$ and
$g_n\to g$ in $\left(\int_{\Omega}^{\oplus}B_\omega\,d\mu\right)_K$. 
Thanks to Proposition \ref{prop:strong-ae}, we can assume that
$f_n\to f$ a.e. and
$g_n\to g$ a.e.
Due to Definition \ref{def:operators1}, $P(\omega)f_n(\omega)\to P(\omega)f(\omega)$ a.e. and therefore $g(\omega) = P(\omega)f(\omega)$.
Thus, by the closed graph theorem, $P$ is bounded.  

Next, we show that $N(P)(\omega)$ is essentially bounded. 
Take any $0<c<\esssup\limits_{\Omega} N(P)(\omega)$.
There exists a measurable set $A$ s.t. $c\leq N(P)(\omega)$ on $A$.
For any $\varepsilon>0$, there exists a measurable section $v(\omega)$ s.t. 
$(1-\varepsilon)c \leq \frac{\|P(\omega)v(\omega)\|_\omega}{\|v(\omega)\|_\omega}$ on $A$.
Define $s(\omega):= \frac{v(\omega)}{\|v(\omega)\|_\omega}\cdot\chi_A(\omega)$, then $s\in \left(\int_{\Omega}^{\oplus}B_\omega\,d\mu\right)_K$,
$\|s\|_{K(\Omega, \{B_\omega\})}=\|\chi_A\|_{K}$, and 
$\|P(\omega)s(\omega)\|_\omega \geq (1-\varepsilon)c\chi_A(\omega)$ a.e. .
Then
\begin{multline*}
\|P\|\cdot\|\chi_A\|_K = \|P\|\cdot\|s\|_{K(\Omega,\{B_\omega\})}
\geq \|Ps\|_{K(\Omega,\{B_\omega\})} \\
= \big\| \|P(\omega)s(\omega)\|_\omega \big\|_K
\geq (1-\varepsilon)c\|\chi_A\|_K.  
\end{multline*}
Thus, $\|P\|\geq c$ and $\|P\| \geq \esssup\limits_{\Omega} N(P, \omega)$.

Now let $N(P) \in L^{\infty}(\Omega)$.
By Lemma \ref{lemma:NsupK} for any $f\in \left(\int_{\Omega}^{\oplus}B_\omega\,d\mu\right)_K$ the inequality
$\|P(\omega)f(\omega)\|_{\omega} \leq  \|N(P)\|_{L^{\infty}(\Omega)} \|f(\omega)\|_{\omega}$ holds a.e.,
therefore, 
\[
\big\|\|P(\omega)f(\omega)\|_{\omega}\big\|_K \leq  \|N(P)\|_{L^{\infty}(\Omega)} \big\|\|f(\omega)\|_{\omega}\big\|_K.
\]
\end{proof}
The above proof generalizes the solution of \cite[Problem 66]{Halmos1982} (boundedness of multiplications) from $L^2$ to $K$-direct integrals.

\subsection{Decomposable operators between distinct families }
As it was mentioned in the introduction, in this article we apply ideas, rising to the classical theory of decomposable operators, to a different type questions, related to evolutionary problems and the composition operator on Banach spaces. For this purpose, we generalize the previously introduced concept.

Let $\{D_\omega\}$ be another measurable family of Banach spaces over 
$(\Omega, \nu, W)$ ($\nu$ is absolutely continuous with respect to $\mu$), originating from the measurable family of semi-norms $\{\|\cdot\|_{D_\omega}\}_{\omega\in\Omega}$ on a vector space $W$.

\begin{defn}\label{def:operatorsQ}
\textit{A measurable family of bounded linear operators} $\{Q(\omega)\}_{\omega \in \Omega}$ is a map $\omega\mapsto Q(\omega)\in\mathcal{L}(B_{\omega},D_{\omega})$, defined for $\mu$-almost all $\omega \in \Omega$, such that
\begin{enumerate}
\item A map $\omega\mapsto Q(\omega)v$ is a measurable section of $\{D_{\omega}\}_{\omega\in \Omega}$ for any $v\in V$;
\item $\|Q(\omega)\|_{\mathcal B(B_\omega, D_\omega)} <\infty$ $\mu$-almost everywhere on $\Omega$.
\end{enumerate}
\end{defn}
Then we can define the random norm
$$
N(Q)(\omega) := \bigvee \left\{ \frac{\|Q(\omega)v\|_{D_\omega}}{\|v\|_{B_\omega}} : v\in V\right\}.
$$
It is easy to prove, as in Lemma \ref{lemma:NsupK}, that the lattice supremum in the definition of $N(Q)$ can be taken over all measurable section of family $\{B_\omega\}$.

\begin{defn}
A linear operator $Q : \int_{\Omega}^{\oplus}B_\omega\,d\mu \to \int_{\Omega}^{\oplus}D_\omega\,d\nu$
is said to be \textit{decomposable} if it is defined by the measurable family of bounded operators 
$\{Q(\omega)\}$, in the sense that
\begin{equation}\label{def:M}
Q[f](\omega) = Q(\omega)[f(\omega)]
\end{equation}
for every $f\in \int_{\Omega}^{\oplus}B_\omega\,d\mu$ (see Fig. \ref{fig:operatorQ}).
\end{defn}

\begin{figure}[h]
\centering	
\includegraphics[width=0.8\linewidth]{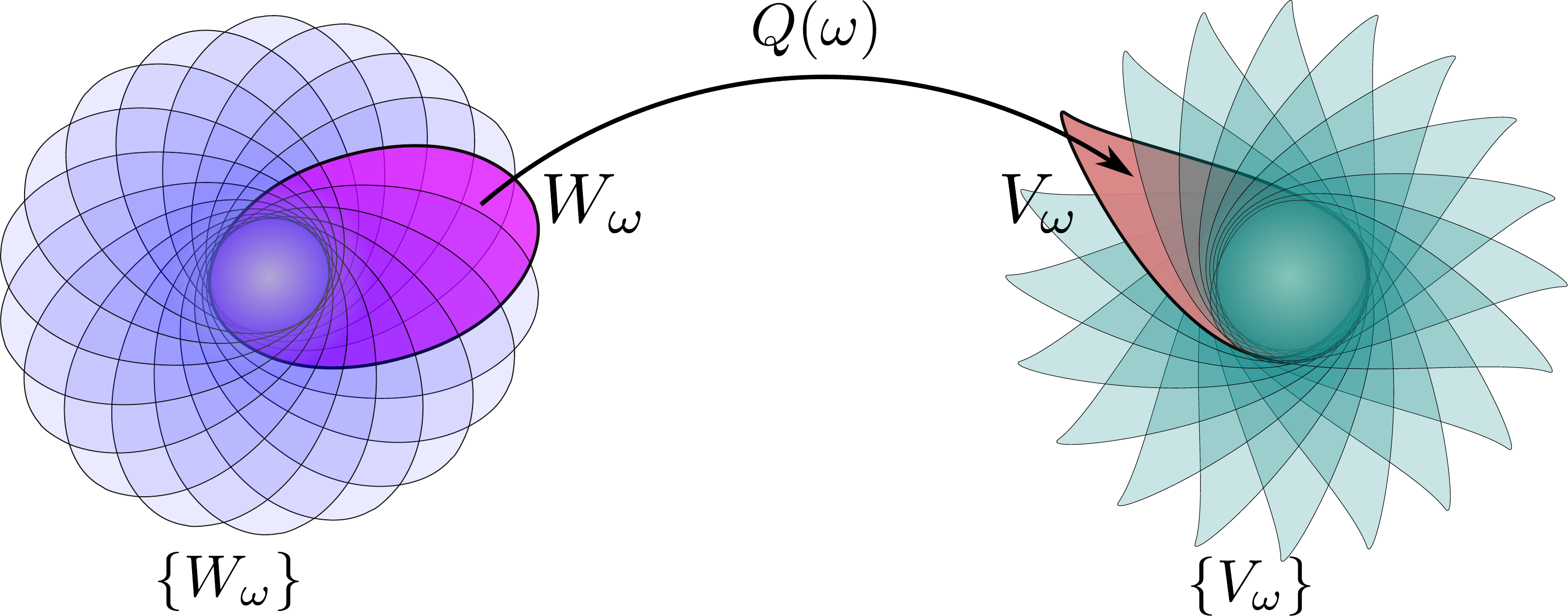}
\caption{Operator $Q$ acts between different measurably families} \label{fig:operatorQ}
\end{figure}

In article \cite{JR17}, a simple example of constructing of a decomposable operator is provided using the "lifting" of operator, acting on the core vector space. In the same way, we can construct an example of a generalized decomposable operator.
Let $\widetilde Q : V \to W$ be a linear mapping such that there exist a function $M \in L^\infty(\Omega)$ and the inequality $\|\widetilde Qv\|_{D_\omega} \leq M(\omega) \|v\|_{B_\omega}$, $v \in V$, holds for $\mu$-almost all $\omega \in \Omega$. Since $\widetilde Q$ maps $\ker(\|\cdot\|_{B_\omega})$ in $\ker(\|\cdot\|_{D_\omega})$, then there exists a linear map $Q(\omega): V / \ker(\|\cdot\|_{B_\omega}) \to W / \ker(\|\cdot\|_{D_\omega})$, defined for $\mu$-almost all $\omega \in \Omega$. This operator can be extended to a bounded operator $Q(\omega): B_\omega \to D_\omega$. 
It is easy to see that $\omega \mapsto Q(\omega)v$ is a measurable section of $\{D_{\omega}\}_{\omega\in \Omega}$, $\|Q(\omega)\|_{\mathcal B(B_\omega, D_\omega)} <\infty$ almost everywhere on $\Omega$ due to the construction of operators, and, hence, $\{Q(\omega)\}_{\omega \in \Omega}$ is a measurable family of operators.
Finally, it is clear that $N(Q)(\omega) \leq M(\omega)$ a.e.

As in Theorem \ref{theorem:MultiplicationOperator}, we have that the family $\{Q(\omega)\}$ induces an 
operator $Q: \left(\int_{\Omega}^{\oplus}B_\omega\,d\mu\right)_K \to \left(\int_{\Omega}^{\oplus}D_\omega\,d\nu\right)_K$ iff $N(Q)\in L^\infty(\Omega)$ and in this case $Q$ is a bounded operator
with norm $\|Q\| = \esssup\limits_{\Omega} N(Q)(\omega)$.


The next natural step is to consider different external K\"ote norms over direct integrals, i.e. $Q: \left(\int_{\Omega}^{\oplus}B_\omega\,d\mu\right)_{K_1} \to \left(\int_{\Omega}^{\oplus}D_\omega\,d\nu\right)_{K_2}$. Unfortunately we cannot provide characterization in terms of $N(Q)$ for general case (see Remark \ref{remsetfunction}), but, in the particular case, when $K_1$ and $K_2$ are $L_p$-spaces and $\Omega$ is homogeneous space, we have the following characterization.
\begin{thm}\label{theorem:operatorQ}
Let $1\leq q<p<\infty$ and  measure $\mu$ be absolutely continuous with respect to $\nu$.
Then a measurable family of operators $\{Q(\omega)\}$ induces a bounded operator from 
$\left(\int_{\Omega}^{\oplus}B_\omega\,d\mu\right)_{L^p}$
to
$\left(\int_{\Omega}^{\oplus}D_\omega\,d\nu\right)_{L^q}$ by the rule $Qf(\omega) = Q(\omega)f(\omega)$
iff
$$
N(Q)J \in L^{\frac{pq}{p-q}}(\Omega),
$$
where $J = \frac{d\nu}{d\mu}$ is a Radon--Nikod\'ym derivative.
\end{thm}
The proof of this theorem will be provided in Section \ref{section:WeightedCO}.

\section{Decomposable operators on distinct spaces (Construction of $M_F$) }\label{section:MFoperators}
In this section, we introduce the concept of a decomposable operator acting between different $L^p$-direct integrals. This concept, on the one hand, corresponds to the concepts of the previous section, on the other hand, it reflects the idea of considering the composition operator, defined not by a mapping, but by a binary correspondence. Unfortunately, due to the large number of objects involved, the formulations and proofs look cumbersome and too technical.

Let $(T, \Sigma_T, \mu)$ and $(S, \Sigma_S, \nu)$ be measure spaces and $\{W_t\}_{t\in T}$, $\{V_s\}_{s\in S}$ be corresponding measurable families of Banach spaces over vector spaces $W$ and $V$ respectively.
Choose a measurable subset $F$ of the product measurable space $S\times T$ and consider a measure space $(F, \Sigma_F,\lambda)$, where $\Sigma_F$ is generated by $\Sigma_S \times \Sigma_T $, and $\lambda$ is a measure. 
Take an additional family of Banach spaces  $\{\widetilde{V}_{(s,t)}\}_{(s,t)\in F}$, such that $\widetilde V_{(s,t)} = V_s$ for all $(s,t)\in F$.
Note that this family is measurable due to its construction (also see \cite[Lemma III.11.10]{DunfordSchwartz1988}).
Next, let $\int_{T}^{\oplus} W_t\,d\mu$ and $\int_{F}^{\oplus} \widetilde V_{(s,t)} \, d\lambda$ be direct integrals over these families with respect to measures $\mu$ and $\lambda$.

\begin{defn}\label{def:operatorsF}
\textit{A measurable family of bounded linear operators} $\{P(s,t)\}_{(s,t)\in F}$ is a map $(s,t) \mapsto P(s,t) \in \mathcal{L}(W_t,V_s)$, defined for $\lambda$-almost all $(s,t)\in F$, such that
\begin{enumerate}
\item A map $(s,t) \mapsto P(s,t)w$ is a measurable section of $\{\widetilde{V}_{(s,t)}\}_{(s,t)\in F}$ for any $w\in W$;
\item $\|P(s,t)\|_{\mathcal B(W_t,V_s)} <\infty$ almost everywhere on $F$.
\end{enumerate}
\end{defn}

\begin{lem}
Let $\{P(s,t)\}_{(s,t)\in F}$ be a measurable family of bounded linear operators,
then for any measurable section $\zeta(t)$ of $\{W_t\}_{t\in T}$ $P(s,t)\zeta(t)$ is a measurable section of $\{\widetilde{V}_{(s,t)}\}_{(s,t)\in F}$.
\end{lem}
\begin{proof}
The same as in Lemma \ref{lemma:MeasurableSection}.
\end{proof}

Let $\mathcal M$ be the set of all measurable families of bounded linear operators as in Definition \ref{def:operatorsF}.
Define a mapping $N:\mathcal M\to L^0(F,\lambda)$ which maps a measurable family to a measurable function by the rule
$$
N(P)(s,t) = \bigvee \left\{ \frac{\|P(s,t)w\|_s}{\|w\|_t} : w\in W\right\},
$$
with the convention that $\frac{\|P(s,t)w\|_s}{\|w\|_t} = 0$ whenever $\|w\|_t = 0$.

\begin{lem}\label{lemma:NsupKst}
The value of $N(P)(s,t)$ will not change if the lattice supremum is taken over all measurable sections of $\{W_t\}$, i.e.
\begin{equation}\label{eq:Nsupst}
N(P)(s,t) = \bigvee \left\{ \frac{\|P(s,t)f(t)\|_s}{\|f(t)\|_t} : f \text{ is a measurable section of } \{W_t\} \right\}.
\end{equation}
\end{lem}
\begin{proof}
The same as in Lemma \ref{lemma:NsupK}.
\end{proof}

\begin{defn}
We say that operator 
$$
M_F: \int_{T}^{\oplus}W_t\,d\mu \to \int_{F}^{\oplus} \widetilde{V}_{(s,t)}\,d\lambda
$$
is \textit{decomposable operator between distinct spaces} if
there exists a measurable family of bounded linear operators $P$ such that
\begin{equation}\label{def:MF}
M_F[f](s,t) = P(s,t)[f(t)].
\end{equation}
\end{defn}

\begin{figure}[h]
\centering	
\includegraphics[width=0.8\linewidth]{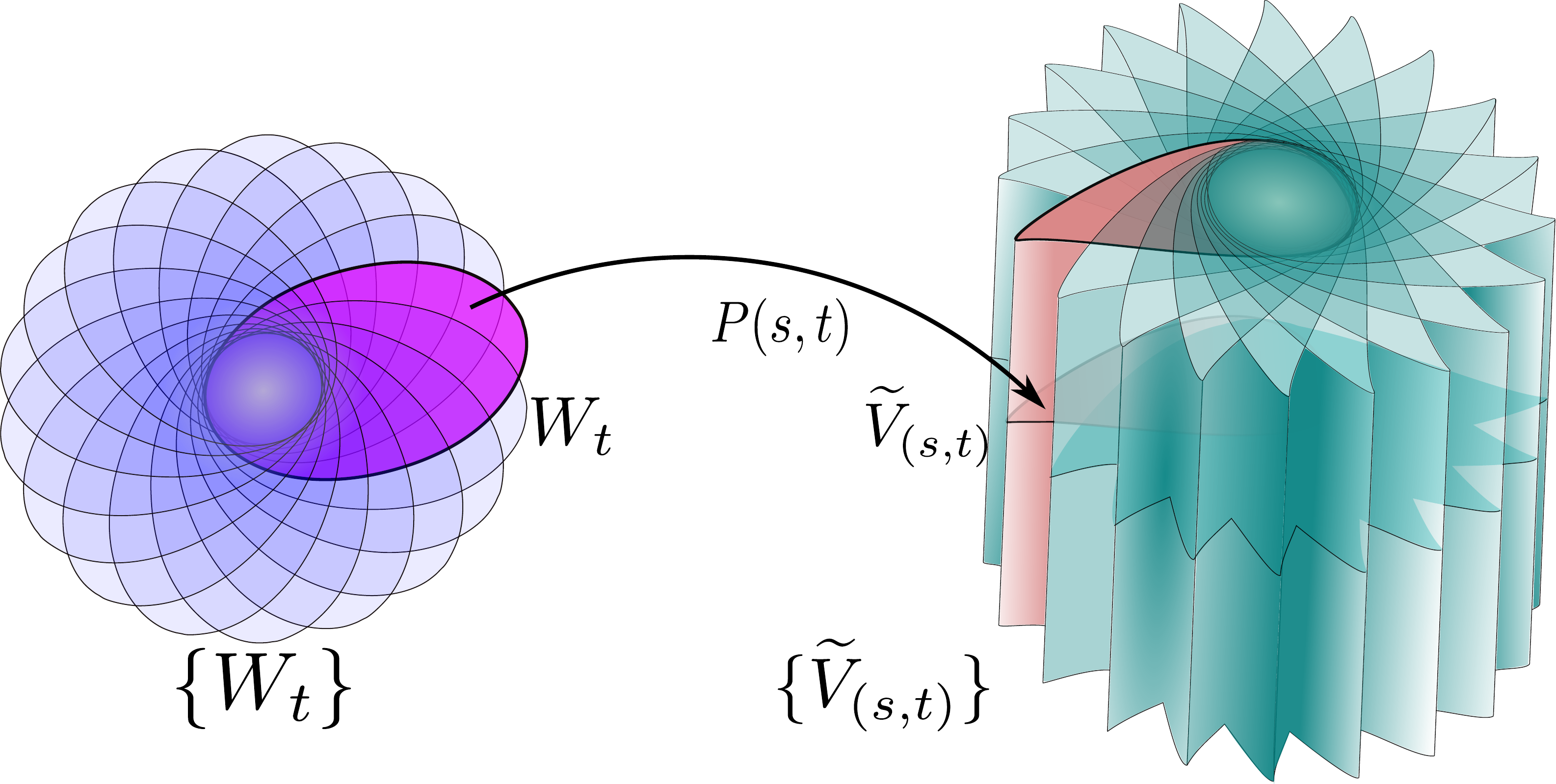}
\caption{Operator $M_F$ acts from a measurable family $\{W_t\}$ to a "cylinder" $\widetilde{V}_{(s,t)}$ with the base $\{V_s\}$} \label{fig:2}
\end{figure}

\begin{lem}
 The mixed operator $M_F$ is well-defined, in the sense that if $f \in \int_{T}^{\oplus}W_t\,d\mu$ then $M_F[f](s,t) \in \int_{F}^{\oplus}\widetilde{V}_{(s,t)}\,d\lambda$.
\end{lem}

\begin{proof}
We have to prove that measurability (in the sense of sections) of $P \in \int_{F}^{\oplus} \mathcal{B}(W_t, V_s) \, d\lambda$ and $f \in \int_{T}^{\oplus} W_t \, d\mu$ implies measurability of function $M_F[f](s,t) = g(s,t) \in \int_{F}^{\oplus} \widetilde{V}_{(s,t)} \, d\lambda $. Let $\{H_n\}_{n \in \mathbb{N}}$ be a sequence of simple sections in $\int_{F}^{\oplus} \mathcal{B}(W_t, V_s) \, d\lambda$ which converges to $P$ for almost all $(s,t) \in F$ and $\{h_n\}_{n \in \mathbb{N}}$ be the same sequence for function $f$. Then consider the sequence $H_n[h_n]$. Due to construction, elements of this sequence are simple sections in  $\int_{F}^{\oplus} \widetilde{V}_{(s,t)} \, d\lambda$. Let us show that this sequence converges to $g(s,t)$ for almost all $(s,t) \in F$.
\begin{multline*}
  \lim\limits_{n \to \infty} \| g(s,t) - H_n(s,t)[h_n(t)] \|_s \\
  = \lim\limits_{n \to \infty} \| P(s,t)[f(t)] + P(s,t)[h_n(t)] - P(s,t)[h_n(t)] - H_n(s,t)[h_n(t)] \|_s \\
  \leq \lim\limits_{n \to \infty} \| P(s,t)[f(t)] - P(s,t)[h_n(t)] \|_s + \lim\limits_{n \to \infty} \| P(s,t)[h_n(t)] - H_n(s,t)[h_n(t)] \|_s.
\end{multline*}
By virtue of linearity and boundedness of $P(s,t)$ in the first summand, the sequence $P(s,t)[h_n(t)]$ tends to $P(s,t)[f(t)]$ and, hence, the limit is zero. Consider the second summand.
\begin{multline*}
  \lim\limits_{n \to \infty} \| P(s,t)[h_n(t)] - H_n(s,t)[h_n(t)] \|_s \\ 
  = \lim\limits_{n \to \infty} \frac{\| P(s,t)[h_n(t)] - H_n(s,t)[h_n(t)] \|_s}{\|h_n(t)\|_t} \|h_n(t)\|_t  \\
  \leq \lim\limits_{n \to \infty} \sup\limits_{\|w\|_t \ne 0} \frac{\| P(s,t)[w] - H_n(s,t)[w] \|_s}{\|w\|_t} \|h_n(t)\|_t.
\end{multline*}
The first multiplier tends to zero due to definition of $H_n$ and the second is finite. Hence, we yield that $H_n[h_n]$ converges to $g(s,t)$ for almost all $(s,t) \in F$.
\end{proof}

\subsection{Boundedness}
Here we are going to study the following question: \textit{under which conditions is the operator $M_F$ bounded on $L^p$-direct integrals?} 
It means that for the operator 
$$
M_F:\left(\int_{T}^{\oplus}W_t\,d\mu\right)_{L^p} \to \left(\int_{F}^{\oplus}\widetilde{V}_{(s,t)}\,d\lambda\right)_{L^q}
$$
there is a constant $K$ such that
$$
\|M_F f\|_{L^q(F,\{\widetilde{V}_{(s,t)}\})} \leq K \|f\|_{L^p(T,\{W_t\})}
$$
for all $f\in \left(\int_{T}^{\oplus}W_t\,d\mu\right)_{L^p}$. 
Hereinafter we assume that spaces $(T, \mu)$ and $(S, \nu)$ are the homogeneous spaces with metrics $d$ and $\rho$ respectively.

We begin with a lemma, which plays a key role in the proof of the main result (the necessary part). It describes the boundedness of operator $M_F$  through a set function of special type.
This kind of result was originated by the works of S.~Vodopyanov and A.~Ukhlov (e.g. \cite{U93,VU10,VU04}), where the boundedness of composition operators on Lebesgue and Sobolev spaces was studied.
In our setting it allows to obtain the dependence of the optimal constant $K$ on the domain of definition:
$$
\|M_F f\|_{L^q(F\cap S\times A,\{\widetilde{V}_{(s,t)}\})} \leq \Phi(A) \|f\|_{L^p(A,\{W_t\})},
$$
for all $f\in \left(\int_{A}^{\oplus}W_t\,d\mu\right)_{L^p}$, where $A\subset T$.

\begin{lem}\label{lemma:setfunction1}
Suppose that operator $M_{F}$ is bounded,  
then
$$
\Phi(A) = \sup\limits_{f\in L^p(A, \{W_t\})}
\left(\frac{\|M_{F} f\|_{L^q(F\cap S\times A,\{\widetilde{V}_{(s,t)}\})}}{\|f\|_{L^p(A,\{W_t\})}} \right)^{\kappa}, 
\quad \kappa = \frac{pq}{p-q}
$$
is a bounded monotone countable additive function on all Borel sets $A\subset T$, ${\mu(A)>0}$.  
\end{lem}   

\begin{proof}
From the very definition of the set function $\Phi$ we immediately conclude 
$\Phi(A_1)\leq\Phi(A_2)$ whenever $A_1\subset A_2$, in particular 
$\Phi(A)\leq \Phi(T)\leq \|M_F\|^{\kappa}$.
Thus, properties of boundedness and monotonicity hold.

Now show the countable additivity of $\Phi$.
Let $\{A_i\}_{i\in\mathbb N}$ be disjoint Borel sets in $T$.
Fix $\varepsilon>0$. 
The supremum property implies there exist functions $\{f_i\}$ in $L^p(T, \{W_t\})$
such that
\begin{equation}\label{eq:sup_prop}
\|M_F f_i\|_{L^q(F\cap S\times A_i,\{\widetilde{V}_{(s,t)}\})}
\geq \left(\Phi(A_i)\left(1 - \frac{\varepsilon}{2^i}\right) \right)^{\frac{1}{\kappa}}
\|f_i\|_{L^p(A_i, \{W_t\})}.
\end{equation}
Homogeneity of the $L^p$ norm and operator $M_F$ allow us to choose functions $\{f_i\}$
in such a way that
\begin{equation}\label{eq:hom_prop}
\|f_i\|^p_{L^p(A_i, \{W_t\})} = \Phi(A_i)\left(1 - \frac{\varepsilon}{2^i}\right).
\end{equation}
We can assume $f_i=0$ on  $T\setminus A_i$.
Then for an arbitrary $N\in\mathbb N$, we use linearity of the operator
and apply inequality \eqref{eq:sup_prop}.
\begin{multline*}
\bigg\|M_F \sum\limits_{i=1}^N f_i\bigg\|^q_{L^q(\bigcup F\cap S\times A_i,\{\widetilde{V}_{(s,t)}\})} 
= \sum\limits_{i=1}^N \|M_F f_i\|^q_{L^q(F\cap S\times A_i,\{\widetilde{V}_{(s,t)}\})}\\
\geq \sum\limits_{i=1}^N\left(\Phi(A_i)\left(1 - \frac{\varepsilon}{2^i}\right) \right)^{\frac{q}{\kappa}}
\|f_i\|^q_{L^p(A_i, \{W_t\})}. 
\end{multline*}  

Applying equality  \eqref{eq:hom_prop} twice, one obtains

\begin{multline*}
\sum\limits_{i=1}^N\left(\Phi(A_i)\left(1 - \frac{\varepsilon}{2^i}\right) \right)^{\frac{q}{\kappa}}
\|f_i\|^q_{L^p(A_i, \{W_t\})} 
= \sum\limits_{i=1}^N\left(\|f_i\|_{L^p(A_i, \{W_t\})} \right)^{\frac{pq}{\kappa}}
\|f_i\|^q_{L^p(A_i, \{W_t\})}\\
= \sum\limits_{i=1}^N\|f_i\|^p_{L^p(A_i, \{W_t\})}
=\left(\sum\limits_{i=1}^N\|f_i\|^p_{L^p(A_i, \{W_t\})}\right)^{\frac{p-q}{p}}\left(\sum\limits_{i=1}^N\|f_i\|^p_{L^p(A_i, \{W_t\})} \right)^{\frac{q}{p}}\\
= \left(\sum\limits_{i=1}^N\Phi(A_i)\left(1 - \frac{\varepsilon}{2^i}\right)\right)^{\frac{p-q}{p}}\left(\sum\limits_{i=1}^N\|f_i\|^p_{L^p(A_i, \{W_t\})} \right)^{\frac{q}{p}}\\
\geq \left(\sum\limits_{i=1}^N\Phi(A_i) - \varepsilon\Phi(\bigcup A_i) \right)^{\frac{p-q}{p}}
\bigg\|\sum\limits_{i=1}^N f_i\bigg\|^q_{L^p(\bigcup A_i, \{W_t\})}.
\end{multline*}
Thus, 
$$
\left(\frac{\big\|M_F \sum\limits_{i=1}^N f_i\big\|_{L^q\big(\bigcup F\cap S\times A_i,\{\widetilde{V}_{(s,t)}\} \big)}}
{\big\|\sum\limits_{i=1}^N f_i\big\|_{L^p\big(\bigcup A_i, \{W_t\} \big)}} \right)^{\kappa}
\geq \sum\limits_{i=1}^N\Phi(A_i) - \varepsilon\Phi\big(\bigcup A_i\big),
$$
and it follows
$$
\Phi\big(\cup A_i\big) \geq \sum\limits_{i=1}^N\Phi(A_i).
$$
Check the reverse inequality. With the help of  H\"older's inequality have 
\begin{multline*}
\big\|M_F f\big\|^q_{L^q(\bigcup F\cap S\times A_i,\{\widetilde{V}_{(s,t)}\})} 
= \sum\limits_{i=1}^N \|M_F f\|^q_{L^q(F\cap S\times A_i,\{\widetilde{V}_{(s,t)}\})}\\
\leq \sum\limits_{i=1}^N\left(\Phi(A_i) \right)^{\frac{q}{\kappa}}\|f_i\|^q_{L^p(A_i, \{W_t\})}
\leq \left(\sum\limits_{i=1}^N\Phi(A_i)\right)^{\frac{p-q}{p}}\left(\sum\limits_{i=1}^N\|f_i\|^p_{L^p(A_i, \{W_t\})}\right)^{\frac{q}{p}} . 
\end{multline*}  
Consequently,
$$
\frac{\big\|M_F f\big\|_{L^q(\bigcup F\cap S\times A_i,\{\widetilde{V}_{(s,t)}\})}}{\|f\|_{L^p(\cup A_i, \{W_t\})}} \leq \left(\sum\limits_{i=1}^N\Phi(A_i)\right)^{\frac{1}{\kappa}}
$$
and
$
\Phi\big(\cup A_i\big) \leq \sum\limits_{i=1}^N\Phi(A_i)
$
as desired.
\end{proof}

\begin{rem}\label{remsetfunction}
Set functions of this type can be examined for a general case, when K\"ote spaces $E_1$, $E_2$ are considered instead of Lebesgue spaces. 
Unfortunately, it seems that the lemma is valid only in the case of $L^p$-norms, so we cannot prove the additivity property.  
There are some attempts to obtain same result for other spaces, but it demands a number of additional assumptions (e. g. \cite[Lemma 1]{M17}, \cite{E15,E17}).
Therefore, this is a main technical obstacle that leads us to consider only the case of $L^p$-spaces.
\end{rem}

The first main result of the article is the following theorem.

\begin{thm}\label{thm:F1}
Let $\lambda$ be absolutely continuous with respect to $\mu\times\nu$ and 
$\{P(s,t)\}_{(s,t) \in F}$ be a measurable family of operators. Then the operator 
$M_F:\left(\int_{T}^{\oplus}W_t\,d\mu\right)_{L^p} \to \left(\int_{F}^{\oplus} \widetilde{V}_{(s,t)} \,d\lambda\right)_{L^q}$, $p\geq q$, 
is bounded if and only if
\begin{equation}\label{thm:main:eq} 
N(P)(s,t)\cdot J^{\frac{1}{q}}(s,t)\in L^{\kappa, q}(F), \quad \kappa = \begin{cases} \frac{pq}{p-q}, & p>q, \\ \infty, & p=q, \end{cases}
\end{equation}
$J(s,t) = \frac{d\lambda}{d\mu\times\nu}$ is a Radon--Nikod\'ym derivative of measure $\lambda$ with respect to $\mu\times\nu$, 
and $\kappa = \frac{pq}{p-q}$ if $p > q$ and $\kappa = \infty$ if $p=q$.
\end{thm}

\begin{proof}
Let $p>q$.
Thanks to Lemma \ref{lemma:setfunction1} for any function $f \in L^p(T, \{W_t\}))$ and a ball $B(\tau,r) \subset T$ we have the next inequality
\begin{equation}\label{thm:main:eq1}
\|M_F f\|_{L^q(F \cap S \times B(\tau,r), \{\widetilde{V}_{(s,t)}\})} \leq \Phi^{\frac{1}{\kappa}}(B(\tau,r)) \|f\|_{L^p(B(t,r),\{W_t\})}.
\end{equation}
Let $u(t)$ be a measurable section of  $\{W_t\}$.
Define section  $g(t) = \frac{u(t)}{\|u(t)\|_t}$.
Note that $\|g\|_{L^p(B(\tau,r); \{W_t\})} = (\mu(B(\tau,r)))^{\frac{1}{p}}$.
Then equation \eqref{thm:main:eq1} implies 
$$
\Bigg(\int_{F \cap S \times B(\tau,r)} \frac{\|P(s, t)u(t)\|^{q}_{s}}{\|u(t)\|^q_t} \, d\lambda(s,t) \Bigg)^{\frac{1}{q}} 
\leq \Phi^{\frac{1}{\kappa}}(B(\tau,r))\mu^{\frac{1}{p}}(B(\tau,r)).  
$$

Using properties of Radon--Nikod\'ym derivative and dividing by $(\mu(B(t,r)))^{\frac{1}{q}}$, we arrive to the following relation
$$
 \Bigg(\frac{1}{\mu(B(\tau,r))}\int_{F \cap S \times B(\tau,r)}\frac{\|P(s, t)u(t)\|^{q}_{s}}{\|u(t)\|^q_t} J(s,t) \, d\mu \times d\nu(t,s) \Bigg)^{\frac{1}{q}} 
\leq \Bigg(\frac{\Phi(B(\tau,r))}{\mu(B(\tau,r))}\Bigg)^{\frac{1}{\kappa}}.  
$$

By the Fubini theorem, 

\begin{multline*}
\int_{S} \Bigg( \frac{1}{\mu(B(\tau,r))} \int_{B(\tau,r)} \chi_{(F \cap S \times B(t,r))} (s,t) \frac{\|P(s, t)u(t)\|^{q}_{s}}{\|u(t)\|^q_t} J(s,t) \, d\mu(t) \Bigg) d\nu(s)\\
\leq \Bigg(\frac{\Phi(B(\tau,r))}{\mu(B(\tau,r))}\Bigg)^{\frac{q}{\kappa}}.  
\end{multline*}

With $r$ tending to $0$ and with help of the Lebesgue theorem, Fatou's lemma, and Proposition \ref{prop:setderivative} we conclude that

$$
\int_{F_\tau} \frac{\|P(s, \tau)u(\tau)\|^{q}_{s}}{\|u(\tau)\|^q_\tau} J(s,\tau) \, d\nu(s)  
\leq \Phi'(\tau) \quad \text{a. e. } \tau \in T.
$$

\begin{figure}[h]
\centering	
\includegraphics[width=0.6\linewidth]{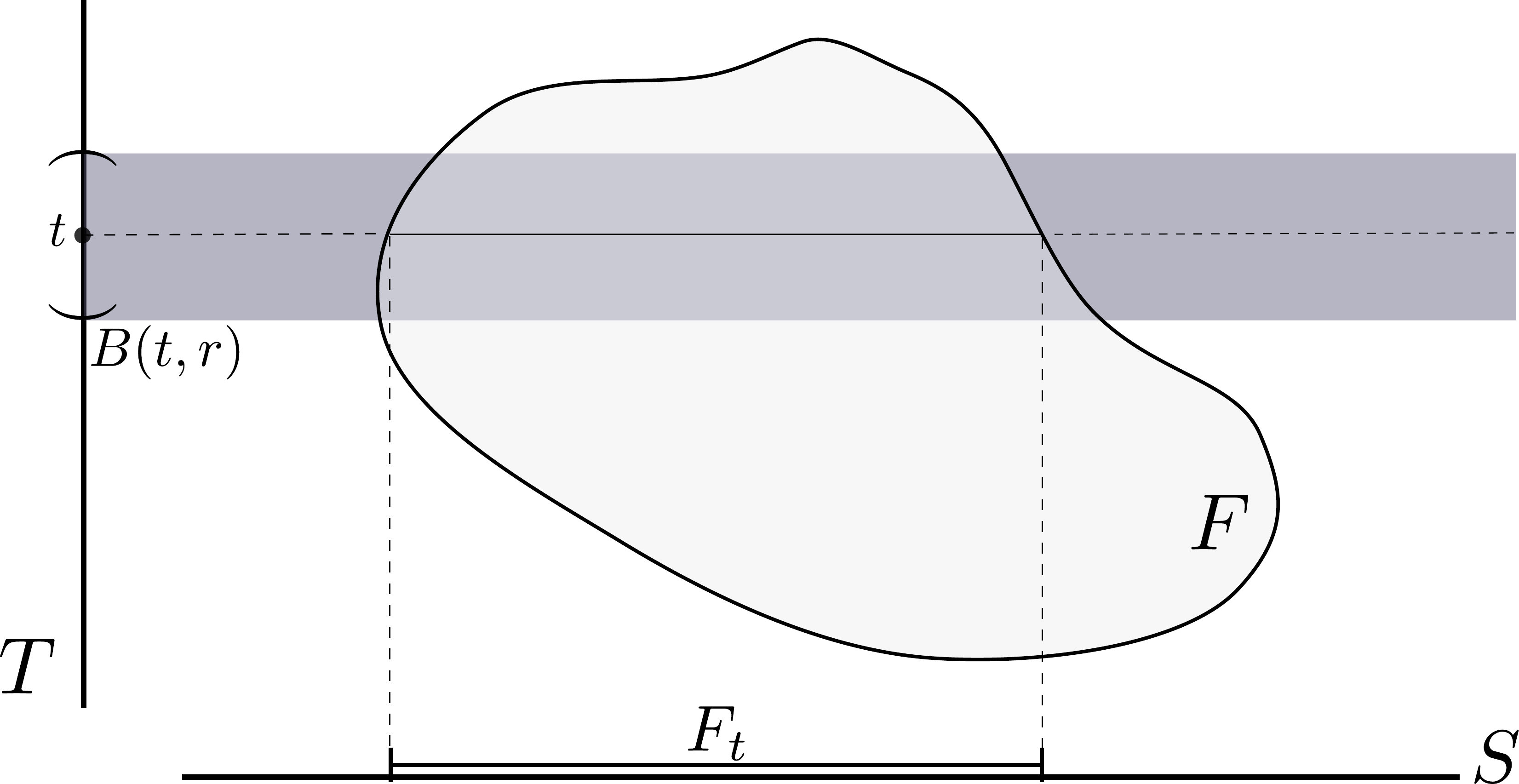}
\caption{Schematic relative arrangement of $B(t,r)$ and $F_t$} \label{fig:3}
\end{figure}

Integrate over $T$ 
$$
\int_{T} \bigg( \int_{F_t} \frac{\|P(s, t)u(t)\|^{q}_{s}}{\|u(t)\|^q_t} J(s,t) \, d\nu(s) \bigg)^{\frac{\kappa}{q}} d\mu(t)  
\leq \Phi(T).
$$

From Proposition \ref{prop:supBFS} and Lemma \ref{lemma:NsupK} we derive 
$$
N(P)(s,t) J^{\frac{1}{q}}(s,t) \in L^{\kappa, q}(F).
$$



Now let $N(P)(s,t)\cdot J^{\frac{1}{q}}(s,t) \in L^{\kappa, q}(F)$. 
As $\|P(s,t)\|_{\mathcal B(W_t,V_s)} <\infty$ almost everywhere on $F$ we have
\begin{multline*}
\|M_{F} f \|_{L^q(F, \{\widetilde{V}_{(s,t)}\})} = \bigg(\int_F \|P(s, t)[f(t)]\|^q_s\, d\lambda(s,t) \bigg)^{\frac{1}{q}} \\
\leq \bigg(\int_F (N(\{P(s,t)\}, s,t) \|f(t)\|_t)^q \, d\lambda(s,t) \bigg)^{\frac{1}{q}} \\
= \bigg( \int_F (N(\{P(s,t)\}, s,t) \|f(t)\|_t)^q J(s,t) \, d\nu \times d\mu(s,t) \bigg)^{\frac{1}{q}}.
\end{multline*}
Using the Fubini theorem and applying H\"older's inequality to the first integral we obtain
\begin{multline*}
\bigg( \int_T \|f(t)\|^q_t \bigg( \int_S \chi_F(s,t) N(\{P(s,t)\}, s,t)^q J(s,t) \, d\nu(s) \bigg) \, d\mu(t) \bigg)^{\frac{1}{q}} \\
\leq \bigg( \int_T \bigg( \int_S \chi_F(s,t) N(\{P(s,t)\}, s,t)^q J(s,t) \, d\nu(s) \bigg)^{\frac{\kappa}{q}} \, d\mu(t) \bigg)^{\frac{1}{\kappa}} \bigg( \int_T \|f(t)\|^p_t \bigg)^{\frac{1}{p}} \\
= \bigg \| N(\{P(s,t)\}, s,t) \cdot J^{\frac{1}{q}}(s,t) \bigg \|_{L^{\kappa, q}(F)} \|f\|_{L^p(T, \{W_t\})}.
\end{multline*}

The case $p = q$ is much simpler, so we will not prove it separately in further results.

For the necessary part we similarly obtain 
$$
\esssup\limits_{T} \int_{S_t} N(\{P(s,t)\}, s,t)^{q} J(s,t) \, d\nu(s) <\infty. 
$$  
For the sufficiency part we do not need to use H\"older's inequality. 

Thus, the proof of the theorem is completed.
\end{proof}

\subsection{Case of neglectable $F$}\label{subsection:graph}

In the next series of theorems, we examine some cases when measure $\lambda$ is not absolutely continuous with respect to $\mu\times\nu$, but it is absolutely continuous with respect to $\mu$. In such cases, we can use as $F$ sets of $\mu\times\nu$-zero measure, e.g. graphs of maps. 



Let us define a Borel measure 
$$
\lambda_P(B) =\int_{F\cap S\times B}\left(N(P)(s,t)\right)^q\,d\lambda ,
$$ for Borel sets $B\subset T$. 

\begin{thm}\label{thm:F2}
Let $\lambda_P$ be absolutely continuous with respect to $\mu$.
Then the operator $M_F:\left(\int_{T}^{\oplus}W_t\,d\mu\right)_{L^p} \to \left(\int_{F}^{\oplus} \widetilde{V}_{(s,t)} \,d\lambda\right)_{L^q}$, $p\geq q$,
is bounded if and only if
$$
J^{\frac{1}{q}}(t)\in L^{\kappa}(T),
$$
where $J(t) = \frac{d\lambda_P}{d\mu}$ is a Radon--Nikod\'ym derivative of measure $\lambda_P$ with respect to $\mu$.
\end{thm}

\begin{proof}
As in the proof of Theorem \ref{thm:F1}, we have the inequality 
$$
\Bigg(\int_{F \cap S \times B(\tau,r)}\frac{\|P(s, t)u(t)\|^{q}_{s}}{\|u(t)\|^q_t} \, d\lambda(s,t) \Bigg)^{\frac{1}{q}} 
\leq \Phi^{\frac{1}{\kappa}}(B(\tau,r))\mu^{\frac{1}{p}}(B(\tau,r)).  
$$
Passing to the lattice supremum (by Proposition \ref{prop:supBFS}), we obtain the same inequality with $N(P)(s,t)$.
Then, by the definition of the measure $\lambda_{P}$ we derive the inequality
$$
\Bigg(\int_{B(\tau,r)}J(t) \, d\mu(t) \Bigg)^{\frac{1}{q}} 
\leq \Phi^{\frac{1}{\kappa}}(B(\tau,r))\mu^{\frac{1}{p}}(B(\tau,r)).  
$$
Repeating the steps of the proof of Theorem \ref{thm:F1}, we obtain the required assertion.
\end{proof}



\section{Weighted composition operator}\label{section:WeightedCO}
Now we obtain the description of the composition operator on $L^p$-direct integrals (and as a consequence on $L^p$-spaces) by making use the construction of Section \ref{subsection:graph}.

We consider the composition operator with operator-valued weight 
induced by a measurable mapping $\psi:S\to T$ 
and a measurable family of bounded linear operators $\{Q(s)\}_{s \in S}$, $Q(s): W_{\psi(s)} \to V_s$ (see Definition \ref{def:operatorsQ}).  
$$
  \mathcal C_\psi: \left(\int_{T}^{\oplus}W_t\,d\mu\right)_{L^p} \to \left(\int_{S}^{\oplus}V_s\,d\nu\right)_{L^q}
$$
and acting by the rule
\begin{equation}\label{Cpsi}
  \mathcal C_\psi[f](s) = Q(s)[f(\psi(s))].
\end{equation}

As mentioned in the introduction we lift operator $\mathcal C_\psi$ to the graph of $\psi$. Then its lifting is the operator 
$M_{\Gamma_\psi}[f](s,t)=P(s,t)[f(t)]$,
where $P(s,t) = Q(s)$ for all $(s,t)\in\Gamma_{\psi}$. 

\begin{lem}\label{lemma:comp1}
For the measurable mapping $\psi:S\to T$ the family $P(s,t) := Q(s)$ for $(s,t)\in \Gamma_\psi$ is measurable 
and $N(P)(s,t) = N(Q)(s)$.
\end{lem}
\begin{proof}
Check Definition \ref{def:operatorsF} for family $\{P(s,t)\}_{(s,t)\in \Gamma_\psi}$.
The measurability holds due to measurability of family $\{\widetilde{V}_{(s,t)}\}_{(s,t)\in \Gamma_\psi}$.
The second property is obvious.

Next, we calculate random norm $N(P)(s,t)$.
For $(s,t)\in \Gamma_\psi$ we have
\begin{multline*}
N(P)(s,t) = \bigvee \left\{ \frac{\|P(s,t)w\|_s}{\|w\|_t} : w\in W\right\}
= \bigvee \left\{ \frac{\|P(s,t)w\|_s}{\|w\|_{\psi(s)}} : w\in W\right\}\\
= \bigvee \left\{ \frac{\|Q(s)w\|_s}{\|w\|_{\psi(s)}} : w\in W\right\}
= N(Q)(s).
\end{multline*}
\end{proof}

Let us additionaly assume that mapping $\psi:S\to T$ enjoys Luzin $\mathcal N^{-1}$-property. For a measurable set $A\subset\Gamma_\psi$ define 
its projection 
$\pi_S(A) = \{s\in S : (s,\psi(s))\in A\}$, which is again measurable. Further, one define measure $\lambda$ on $\Gamma_\psi$ as $\lambda(A) = \nu(\pi_S(A))$.

\begin{lem}\label{lemma:comp2}
Let mapping $\psi:S\to T$ enjoy Luzin $\mathcal N^{-1}$-property.
Then the operator 
$\mathcal C_\psi:\left(\int_{T}^{\oplus}W_t\,d\mu\right)_{L^p} \to \left(\int_{S}^{\oplus}V_s\,d\nu\right)_{L^q}$, $p\geq q$,
is bounded if and only if
operator 
$$
M_{\Gamma_\psi}:\left(\int_{T}^{\oplus}W_t\,d\mu\right)_{L^p} \to \left(\int_{\Gamma_\psi}^{\oplus} \widetilde{V}_{(s,t)} \,d\lambda\right)_{L^q}. 
$$
is bounded. 
Here $M_{\Gamma_\psi}$ is induced by the measurable family 
$P(s,t) = Q(s)$ for $(s,t)\in \Gamma_\psi$.
\end{lem}
\begin{proof}
Check that the operator $E:\left(\int_{S}^{\oplus}V_s\,d\nu\right)_{L^q} \to \left(\int_{\Gamma_\psi}^{\oplus} \widetilde{V}_{(s,t)} \,d\lambda\right)_{L^q}$, 
acting by the rule $E[f](s, t) = f(s)$ is an isomorphism. 
Indeed, for this case, a family of operators $P(s,t)$, associated with $E$, is a family of identical operators $Id : V_s \to V_s$. Hence, $\frac{d\lambda_{Id}}{d\nu} = 1$ for such family. 
Then, due to Theorem \ref{thm:F2}, operator $E$ is bounded.
From the other side, $E$ is a bijection, hence, according to the bounded inverse theorem 
$E^{-1}g(s) = g(s,\psi(s))$ also bounded. 

Now, we prove that  operators $E\mathcal C_\psi$ and $M_{\Gamma_\psi}$ coincide.
For $(s,t)\in \Gamma_\psi$ have
\begin{multline*}
E\mathcal C_\psi[f](s,t) 
= E[Q(\cdot)f(\psi(\cdot))](s,t) 
= E[P(\cdot, t)f(\psi(\cdot))](s,t)\\
= P(s,t)[f(\psi(s))] 
= P(s,t)[f(t)] 
= M_{\Gamma_\psi}[f](s,t).
\end{multline*} 
\end{proof}


\begin{thm}\label{thm:comp1}
Let measure $(N^q(Q)\cdot\nu)\circ\psi^{-1}$ be absolutely continuous with respect to $\mu$.
Then the operator 
$\mathcal C_\psi:\left(\int_{T}^{\oplus}W_t\,d\mu\right)_{L^p} \to \left(\int_{S}^{\oplus}V_s\,d\nu\right)_{L^q}$, $p\geq q$,
is bounded if and only if
\begin{equation}\label{trm:comp1:eq1}
J^{\frac{1}{q}}(t)\in L^{\kappa}(T).
\end{equation}
where $J(t)$ is a Radon--Nikod\'ym derivative of measure 
$(N(Q)\cdot\nu)\circ\psi^{-1}$ with respect to $\mu$.
\end{thm}
\begin{proof}
Thanks to Lemma \ref{lemma:comp2}, we should only investigate the boundedness of
operator $M_{\Gamma_\psi}$. 
By Lemma \ref{lemma:comp1} 
\begin{multline*}
\lambda_P(B) =\int_{F\cap S\times B}\left(N(P)(s,t)\right)^q\,d\lambda
= \int_{\psi^{-1}(B)}\left(N(Q)(s)\right)^q\,d\nu\\ = N^q(Q)\cdot\nu(\psi^{-1}(B)).
\end{multline*}
Then by Theorem \ref{thm:F2}, the boundedness of $M_{\Gamma_\psi}$ is equivalent to 
$$
\left(\frac{d\lambda_P}{d\mu}\right)^{\frac{1}{q}}\in L^{\kappa}(T).
$$
Thus, we obtain \eqref{trm:comp1:eq1},
and the theorem is proved.
\end{proof}

We can give a more explisit form of $J(t)$ in \eqref{trm:comp1:eq1} if $\psi:S\to T$ enjoys some additional conditions. Namely
\begin{cor}\label{cor:comp1}
Let all assumptions of Theorem \ref{trm:comp1:eq1} hold true.
Suppose that $\psi:S\to T$ is an injective mapping and enjoys Luzin $\mathcal N$-property.
Then the operator 
$\mathcal C_\psi:\left(\int_{T}^{\oplus}W_t\,d\mu\right)_{L^p} \to \left(\int_{S}^{\oplus}V_s\,d\nu\right)_{L^q}$, $p\geq q$,
is bounded if and only if
\begin{equation}\label{cor:comp1:eq1}
N(Q)(\psi^{-1}(t)) J^{\frac{1}{q}}_{\psi^{-1}}(t)\in L^{\kappa}(T).
\end{equation}
where $J_{\psi^{-1}}(t)$ is a volume derivative of the inverse mapping.
\end{cor}
\begin{proof}
For $B(t_0, r) \subset T$, let us consider the measure $N^q(Q)\cdot\nu(\psi^{-1}(B(t_0, r)))$. Using change of variable formula \eqref{formula_change} we obtain
\begin{multline*}
N^q(Q)\cdot\nu(\psi^{-1}(B(t_0, r))) = \int_{\psi^{-1}(B(t_0, r))}\left(N(Q)(s)\right)^q\,d\nu \\
= \int_{B(t_0, r)} \left(N(Q)(\psi^{-1}(t))\right)^q J_{\psi^{-1}}(t) \,d\mu
\end{multline*}

From this we conclude that measure $(N^q(Q)\cdot\nu)\circ\psi^{-1}$ is absolutely continuous with respect to $\mu$ and we can apply Theorem \ref{thm:comp1}. Passing to the limit while $r \to 0$ we infer
\begin{multline*}
\left( \frac{d(N^q(Q)\cdot\nu)\circ\psi^{-1}}{d\mu} \right)^{\frac{1}{q}} = \left( \lim\limits_{r \to 0} \frac{\int_{B(t_0, r)} \left(N(Q)(\psi^{-1}(t))\right)^q J_{\psi^{-1}}(t) \,d\mu}{\mu(B(t_0, r))} \right)^{\frac{1}{q}} \\
= N(Q)(\psi^{-1}(t)) J^{\frac{1}{q}}_{\psi^{-1}}(t).
\end{multline*}
Therefore, \eqref{trm:comp1:eq1} takes the form \eqref{cor:comp1:eq1}
\end{proof}

Similarly, as another corollary of Theorem \ref{thm:comp1}, one can easily derive Theorem \ref{theorem:operatorQ}, which was stated in Section \ref{section:Decomposable}
\begin{proof}[Proof of Theorem \ref{theorem:operatorQ}]
Under the condition of the theorem the mapping $\psi: \Omega \to \Omega$ is an identity mapping. Also note that the function $N(Q)(\omega)$ is measurable as a lattice supremum. 
Then, by the Radon--Nikod\'ym theorem, for arbitrary $B(\omega_0, r) \subset \Omega$,
\begin{multline*}
N^q(Q)\cdot\nu(B(\omega_0, r)) = \int_{B(\omega_0, r)}\left(N(Q)(\omega)\right)^q\,d\nu \\
= \int_{B(\omega_0, r)} \left(N(Q)(\omega)\right)^q \frac{d\nu}{d\mu} \,d\mu.
\end{multline*}
Again, applying Theorem \ref{thm:comp1} and passing to the limit while $r \to 0$ completes the proof of the theorem.
\end{proof}

Under additional conditions on the behavior of the norms of the operators $Q(s)$ we can refuse the requirement of the injectivity of the mapping $\psi$.

\begin{thm}\label{thm:comp2}
Let measure $\nu\circ\psi^{-1}$ be absolutely continuous with respect to $\mu$ and 
$N(Q)$ be bounded above and below ($c \leq N(Q)\leq C$). 
Then the operator 
$\mathcal C_\psi:\left(\int_{T}^{\oplus}W_t\,d\mu\right)_{L^p} \to \left(\int_{S}^{\oplus}V_s\,d\nu\right)_{L^q}$, $p\geq q$,
is bounded if and only if
\begin{equation}\label{trm:comp2:eq1}
J_{\psi^{-1}}^{\frac{1}{q}}(t)\in L^{\kappa}(T),
\end{equation}
where $J_{\psi^{-1}}(t)$ is a volume derivative of the inverse mapping.
\end{thm}

\begin{proof}
Thanks to Lemma \ref{lemma:comp2}, we should only investigate the boundedness of
operator $M_{\Gamma_\psi}$.
As in the proof of Theorem \ref{thm:F1}, 
for a ball $B(\tau,r)\subset T$ we have the inequality
$$
\Bigg(\int_{\Gamma_{\psi} \cap S \times B(\tau,r)}\frac{\|P(s, t)u(t)\|^{q}_{s}}{\|u(t)\|^q_t} \, d\lambda(s,t) \Bigg)^{\frac{1}{q}} 
\leq \Phi^{\frac{1}{\kappa}}(B(\tau,r))\mu^{\frac{1}{p}}(B(\tau,r)).
$$
Further, passing to the lattice supremum (by Proposition \ref{prop:supBFS}), in left-hand-side we obtain $N(P)(s,t) = N(Q)(s)$.
Then, due to the definition of $\lambda_{T}$ and boundedness from below of $N(Q)$ we deduce the inequality
$$
c(\lambda_T(B(\tau,r)))^{\frac{1}{q}} 
\leq \Phi^{\frac{1}{\kappa}}(B(\tau,r))\mu^{\frac{1}{p}}(B(\tau,r)).  
$$
Differentiating this inequality as in Theorem \ref{thm:F1}, we obtain
$$
c(J_{\psi^{-1}}(t))^{\frac{\kappa}{q}} \leq \Phi'(t), \quad \text{ a.e. } t\in T,
$$ 
i.e,
\begin{equation}\label{trm:comp2:eq2}
J^{\frac{1}{q}}_{\psi^{-1}} \in L^{\kappa}(T).
\end{equation}

Now, let the inequality \eqref{trm:comp2:eq2} be satisfied. 
Then, using the upper boundedness of the norms of operators and the H\"older inequality, we derive
\begin{multline*}
\|M_{\Gamma_{\psi}} f \|_{L^q(\Gamma_{\psi}, \{\widetilde{V}_{(s,t)}\})} 
= \bigg(\int_{\Gamma_{\psi}} \|P(s, t)[f(t)]\|^q_s\, d\lambda(s,t) \bigg)^{\frac{1}{q}} \\
\leq \bigg(\int_{\Gamma_{\psi}} (N(P)(s,t)\|f(t)\|)^q_s\, d\lambda(s,t) \bigg)^{\frac{1}{q}}
\leq C\bigg(\int_T \|f(t)\|_t^q \, d\lambda_T(t) \bigg)^{\frac{1}{q}} \\
= C\bigg(\int_T \|f(t)\|_t^q J_{\psi^{-1}}(t)\, d\mu(t) \bigg)^{\frac{1}{q}}
\leq C\bigg(\int_T J^{\frac{\kappa}{q}}_{\psi^{-1}}(t)\, d\mu(t) \bigg)^{\frac{1}{\kappa}} \bigg(\int_T \|f(t)\|_t^p \, d\mu(t) \bigg)^{\frac{1}{p}}.
\end{multline*}

\end{proof}

As a corollary of the previous theorem we can obtain the result for Lebesgue--Bochner spaces. We also note that the resulting boundedness condition \eqref{trm:comp2:eq1} coincides with known results for standard Lebesgue spaces (see, for example, \cite[Theorem 4.]{VU04}).

\begin{cor}
Let measure $\nu\circ\psi^{-1}$ be absolutely continuous with respect to $\mu$ and $B$ be a Banach space.
Then the operator 
$\mathcal C_\psi:L^p(T, B) \to L^q(S, B)$, $p\geq q$,
is bounded if and only if
\begin{equation}
J_{\psi^{-1}}^{\frac{1}{q}}(t)\in L^{\kappa}(T),
\end{equation}
where $J_{\psi^{-1}}(t)$ is a volume derivative of the inverse mapping.
\end{cor}
In the case of Lebesgue--Bochner instead $L^p$-direct integrals, a family of operators $\{Q(s)\}$ is reduced to the identical operator $I: B \to B$ and conditions of Theorem \ref{thm:comp1} hold true.

\section{Composition operator on mixed-norm spaces}\label{section:mixed}

In this section, we apply the obtained results to the composition operator $C_\varphi f = f\circ\varphi$, which acts on the mixed-norm spaces.
Theorem \ref{thm:comp1} allows us to study those operators only in the case of mappings of a special form,
that preserve the priority of the variables.

As in precious sections, let $S$, $X$, $T$, $Y$ be homogeneous spaces, $U \subset S\times X$ and $U'\subset T\times Y$.
We say that mapping $\varphi: U \to U'$ is preserving the priority of the variables if $\varphi(s,x) = (\psi(s), \xi(s,x))$
and  $\psi$ enjoys Luzin $\mathcal N^{-1}$-property and $\xi(s,\cdot)$ enjoys Luzin $\mathcal N^{-1}$-property for almost all $s\in S$.
\begin{figure}[h]
\centering	
\includegraphics[width=0.8\linewidth]{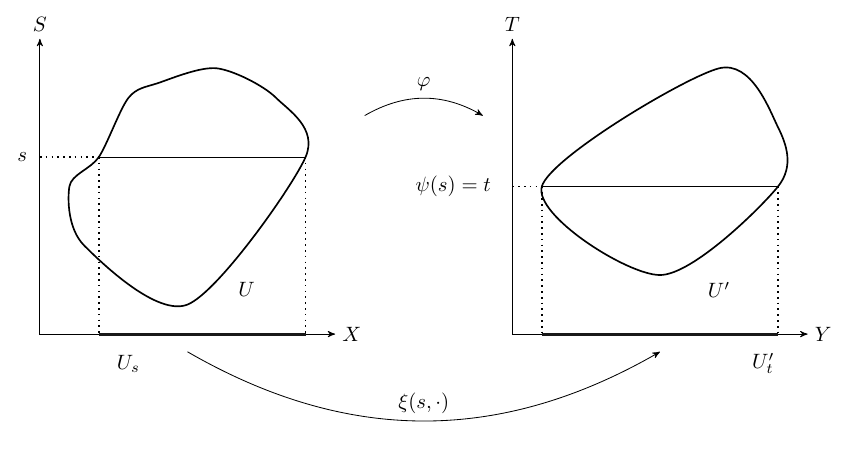}
\caption{Mappings preserving the priority of the variables} \label{fig:6}
\end{figure}

Then the composition operator takes the following form:
\begin{equation}\label{Cphi}
C_\varphi f(s,\cdot) = f(\psi(s),\xi(s,\cdot)) =  Q(s)[f(\psi(s),\cdot)] = \mathcal C_{\psi}[f](s),  
\end{equation}
where $Q(s)$ are composition operators induced by $\xi(s,\cdot):U_s\to U'_{\psi(s)}$.
Because of the construction of the operator $\mathcal C_{\psi}$ (see \eqref{Cpsi}), partial operators $Q(s):L^\beta(U'_{\psi(s)})\to L^\alpha(U_s)$ are assumed to be bounded, this implies additional restrictions on mappings $\xi(s,\cdot)$.

Note, that due to $\mathcal N^{-1}$-property,
for almost every $s\in S$ there exists the volume derivative of the inverse mapping, which is a measurable function: 
\[
J_{\xi^{-1}(s,\cdot)}:U'_{\psi(s)}\to \mathbb R. 
\]

\subsection{The case of Lebesgue spaces}
In \cite{EM2020,EM2019,EM18} we investigated  boundedness of operators on the mixed-norm Lebesgue space, in particular, composition operator and multiplication operator.
Now we can obtain some of our previous results as corollaries from Theorem \ref{thm:comp1} of the previous section.

Consider the question when $\varphi: U \to U'$ induces a bounded composition operator
$$
C_\varphi:L^{p,\beta}(U') \to L^{q,\alpha}(U), \quad \text{ by the rule } C_\varphi f = f\circ\varphi.  
$$

As a direct consequence of Theorem \ref{thm:comp1}, taking into account \eqref{Cphi}, we obtain the next theorem.
\begin{thm}
Let measure $(N^q(Q)\cdot\nu)\circ\psi^{-1}$ be absolutely continuous with respect to $\mu$ and mapping $\varphi: U \to U'$ be such that
\begin{enumerate}
\item[(A1)] $\varphi(s,x) = (\psi(s), \xi(s,x))$ preserves the priority of the variables,
\item[(A2)] $\|J^{\frac{1}{\alpha}}_{\xi^{-1}(s,\cdot)}\|_{L^\gamma(U'_{\psi(s)})} <\infty$ for a.e. $s\in S$ ($\gamma = \frac{\beta\alpha}{\beta-\alpha}$ if $\beta > \alpha$ and $\gamma = \infty$ if $\beta = \alpha$).
\end{enumerate}
Then $\varphi$ induces the bounded composition operator 
$C_\varphi:L^{p,\beta}(U') \to L^{q,\alpha}(U)$, $p\geq q$, $\beta \geq \alpha$,
is bounded if and only if
\begin{equation}
J^{\frac{1}{q}}(t)\in L^{\kappa}(T), \quad \kappa = \begin{cases} \frac{pq}{p-q}, & p>q, \\ \infty, & p=q, \end{cases}
\end{equation}
where $J(t)$ is a Radon--Nikod\'ym derivative of measure 
$(N(Q)\cdot\nu)\circ\psi^{-1}$ with respect to $\mu$, $\kappa = \frac{pq}{p-q}$ if $p > q$ and $\kappa = \infty$ if $p=q$.
Here $N(Q)(s) = \|J^{\frac{1}{\alpha}}_{\xi^{-1}(s,\cdot)}\|_{L^\gamma(U'_{\psi(s)})}$.
\end{thm}
\begin{proof}
By Lemma \ref{mixlem}, mixed-norm Lebesgue space $L^{q,\alpha}(U)$ can be represented as $L^q(S, \{L^\alpha(U_s)\}) = \left(\int_{S}^{\oplus}L^\alpha(U_s)\,d\nu \right)_{L^q}$.
Due to \cite[Theorem 4]{VU04} and (A2) operators $Q(s):L^\beta(U'_{\psi(s)})\to L^\alpha(U_s)$ are bounded composition operators induced by $\xi(s,\cdot)$ and 
$\|J^{\frac{1}{\alpha}}_{\xi^{-1}(s,\cdot)}\|_{L^\gamma(U'_{\psi(s)})}$ are they norms.
To apply Theorem \ref{thm:comp1}, we only should prove that for the case of composition operators 
$N(Q)(s) = \|J^{\frac{1}{\alpha}}_{\xi^{-1}(s,\cdot)}\|_{L^\gamma(U'_{\psi(s)})}$. Let us prove it by Definition \ref{defn:LatticeSup} of the lattice supremum. We know that 
$$
N(Q)(s):= \bigvee \left\{ \frac{\|Q(s)v(s)\|_{L^\alpha(U_s)}}{\|v(s)\|_{L^\beta(U'_{\psi(s)})}} : v \text{ is a measurable section of } \{L^\beta(U'_{\psi(s)})\} \right\}.
$$
Then Condition 1 of Definition \ref{defn:LatticeSup} is satisfied obviously. 

We prove condition 2 by contradiction. Suppose that there exists a measurable function $g(s)$ such that $\frac{\|Q(s)v(s)\|_{L^\alpha(U_s)}}{\|v(s)\|_{L^\beta(U'_{\psi(s)})}} \leq g(s)$ for all measurable section $v(s)$ of $\{L^\beta(U'_{\psi(s)})\}$, but 
$\|J^{\frac{1}{\alpha}}_{\xi^{-1}(s,\cdot)} \|_{L^\gamma(U'_{\psi(s)})} \geq g(s)$ on $s \in \Sigma$, $\nu(\Sigma)>0$.

For function $v(s) = J^{\frac{1}{\beta - \alpha}}_{\xi^{-1}(s,\cdot)}$ with the help of  H\"older's inequality (the case of equality), we obtain
\begin{multline*}
\|Q(s)v(s)\|_{L^\alpha(U_{s})} = \left( \int_{U_{s}} |v(s, \xi(s, x))|^\alpha \, dx \right)^{\frac{1}{\alpha}} \\
= \left( \int_{U'_{\psi(s)}} |v(s, y)|^\alpha J_{\xi^{-1}(s,\cdot)}(s,y) \, dy \right)^{\frac{1}{\alpha}} \\
= \left( \int_{U'_{\psi(s)}} |v(s, y)|^\beta \, dy \right)^{\frac{1}{\beta}} \left( \int_{U'_{\psi(s)}} |J_{\xi^{-1}(s,\cdot)}(y)|^{\frac{\beta}{\beta-\alpha}} \, dy \right)^{\frac{1}{\gamma}},
\end{multline*}
which contradicts our assumption. 
Then, $N(Q)(s) = \|J^{\frac{1}{\alpha}}_{\xi^{-1}(s,\cdot)}\|_{L^\gamma(U'_{\psi(s)})}$.
\end{proof}

If we demand additional condition on the mapping $\varphi: U \to U'$, the statement of the previous theorem can also be formulated in terms of mixed-norm spaces. 
\begin{thm}\label{theorem:mixedL}
Let $p\geq q$ and $\beta\geq\alpha$ and
let mapping $\varphi: U \to U'$ be such that
\begin{enumerate}
\item[(A1)] $\varphi(s,x) = (\psi(s), \xi(s,x))$ preserves the priority of the variables,
\item[(A2)] $\|J^{\frac{1}{\alpha}}_{\xi^{-1}(\psi^{-1}(t),\cdot)}\|_{L^\gamma(U'_{t})} <\infty$ for every $t\in \psi(S)$ ($\gamma = \frac{\beta\alpha}{\beta-\alpha}$ if $\beta > \alpha$ and $\gamma = \infty$ if $\beta = \alpha$),
\item[(A3)] $\psi:S\to T$ is injective and enjoys Luzin $\mathcal{N}$-property.
\end{enumerate}
Then $\varphi$ induces the bounded operator
$C_{\varphi}:L^{p,\beta}(U')\to L^{q,\alpha}(U)$, if and only if
\begin{equation}\label{eq1:theorem:boundedLpq}
J^{\frac{1}{q}}_{\psi^{-1}}(t)J^{\frac{1}{\alpha}}_{\xi^{-1}(\psi^{-1}(t),\cdot)}(y) \in L^{\kappa, \gamma}(U'), \quad \kappa = \begin{cases} \frac{pq}{p-q}, & p>q, \\ \infty, & p=q, \end{cases}
\end{equation}
\end{thm}

\begin{proof}
For the proof it is sufficient to note, that all conditions of Corollary \ref{cor:comp1} are fulfilled. 
Taking into account that 
$N(Q)(s) = \|J^{\frac{1}{\alpha}}_{\xi^{-1}(s,\cdot)} \|_{L^\gamma(U'_{\psi(s)})}$ and combining with \eqref{cor:comp1:eq1}, we obtain the claimed result \eqref{eq1:theorem:boundedLpq}.
\end{proof}


\subsection{The case of Sobolev spaces}

Here we discuss the composition operator on mixed-norm spaces
$$
L^q(S; W^{1,\alpha}(U_s)) = \left(\int_{S}^{\oplus}W^{1,\alpha}(U_s)\,d\nu \right)_{L^q},
$$
where $W^{1,\alpha}(U_s)$ are Sobolev spaces on $U_s$. Such spaces arise as solution spaces while solving the problem of transport of matter in a porous medium (see \cite{MB08}). 
The occurred physical processes are described by a system of two weakly connected parabolic equation: one of them describes the macroscopic level (medium as a whole), and the second one does the microscopic level (processes occurred in a single pore). The porous medium itself is modeled as a set $S$ (the macrolevel), and the distribution of domains $\{U_s\}_{s \in S}$ (the microlevel), see Fig. \ref{fig:7} below. 
The investigation of the composition operator on such spaces can be useful in studying of deformation of a porous medium.

\begin{figure}[h!]
\centering	
\includegraphics[width=0.8\linewidth]{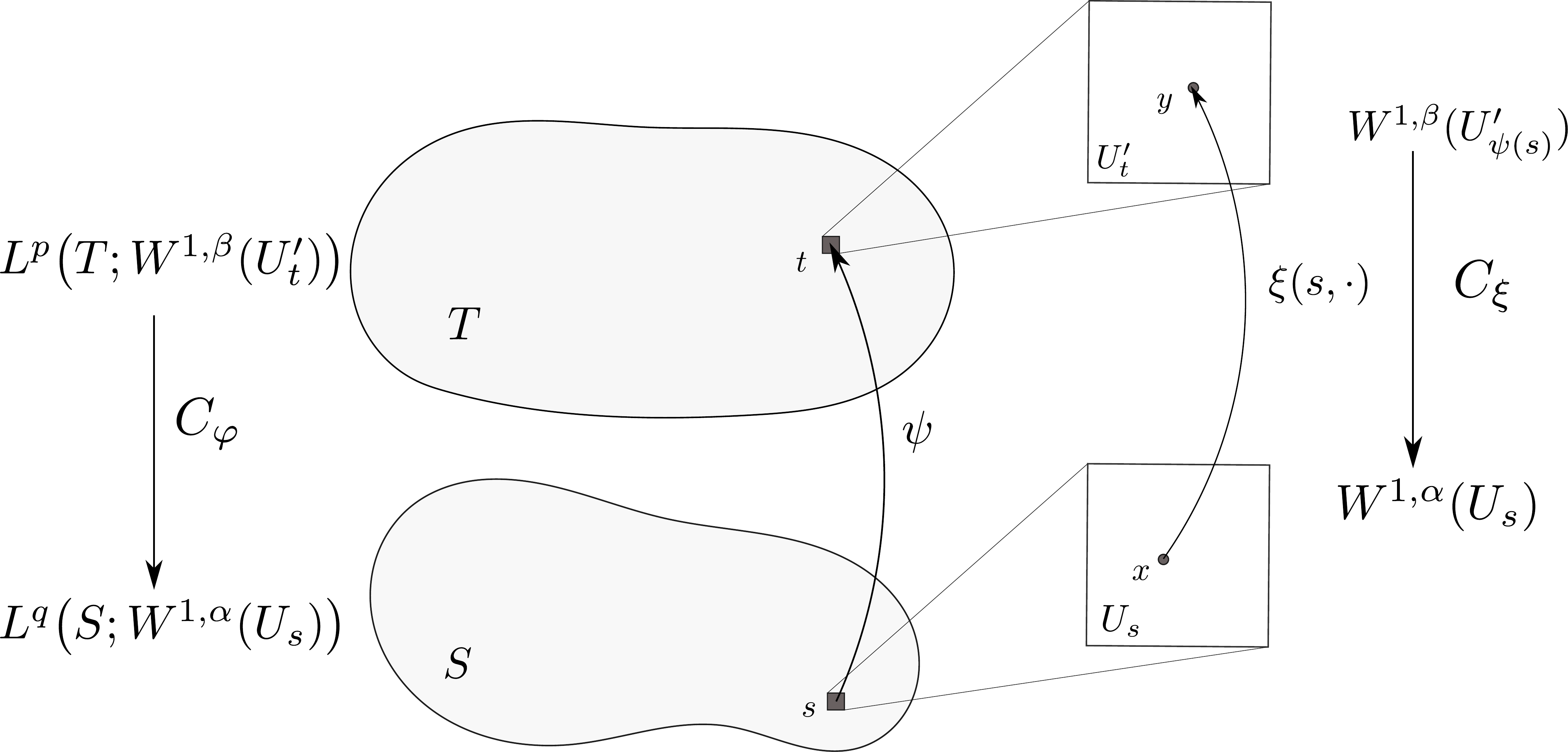}
\caption{Mappings between microstructure models} \label{fig:7}                                          
\end{figure}

As in the previous subsection we have to assume that $Q(s): W^{1,\beta}(U'_{\psi(s)}) \to W^{1,\alpha}(U_{s})$ are bounded composition operators induced by $\xi(s,\cdot)$. 
Note, that this requirement can be viewed as a requirement of preserving internal structure of the investigated medium. 
In the case of Sobolev spaces, the analog of condition (A2) from Theorem \ref{theorem:mixedL} is more complicated, and it is a content of the following theorem:
\begin{thm}[{\cite[Theorem 1.]{VU2002}}]
Let $\Omega$, $\Omega'$ be domains in $\mathbb{R}^n$.
A mapping $\varphi: \Omega \to \Omega'$ generates a bounded composition operator
$$
C_{\varphi}: W^{1,p}(\Omega') \to W^{1,q}(\Omega)
$$
if and only if the following conditions hold
\begin{enumerate}
\item $\varphi  \in ACL(\Omega)$ (absolutely continuous on almost all lines parallel to the coordinate axes);
\item $\varphi$ has a finite distortion ($D \varphi = 0$ a. e. on $\{x\in \Omega : J_{\varphi}(x) = 0\}$, 
$J_{\varphi}$ is the Jacobian determinant of $\varphi$);
\item the distortion function
$$
\mathcal{H}_{\varphi} (y) := \left( \sum_{x \in \varphi^{-1}(y) \setminus \Sigma_{\varphi}} \frac{|D \varphi(x)|^{q}}{|J_{\varphi}(x)|} \right)^{\frac{1}{q}} 
\in L^{\kappa}(\Omega'), \quad \kappa = \begin{cases} \frac{pq}{p-q}, & p>q, \\ \infty, & p=q, \end{cases}
$$
where $\Sigma_{\varphi}$ is a Borel set of measure zero.
\end{enumerate}
The norm of the operator $C_\varphi$ is equivalent to $\|\mathcal{H}_{\varphi}\|_{L^\kappa}(\Omega')$ .
\end{thm}

The class of mappings which satisfied this theorem is known as the mappings with bounded $(p,q)$-distortion. 
Now, let $S$, $T$ be domains in $\mathbb{R}^n$, and $X$, $Y$ be domains $\mathbb{R}^m$.

\begin{thm}
Let measure $(N^q(Q)\cdot\nu)\circ\psi^{-1}$ be absolutely continuous with respect to $\mu$ and mapping $\varphi: U \to U'$ be such that
\begin{enumerate}
\item[(B1)] $\varphi(s,x) = (\psi(s), \xi(s,x))$ preserves the priority of the variables,
\item[(B2)] $\xi(s, \cdot)$ is a mapping with bounded $(\beta, \alpha)$-distortion for every $s \in S$.
\end{enumerate}
Then $\varphi$ induces the bounded composition operator 
$C_{\varphi}:L^p(T, W^{1,\beta}(U'_t)) \to L^q(S, W^{1,\alpha}(U_s))$, $p\geq q$, $\beta\geq\alpha$,
is bounded if and only if
\begin{equation}
J^{\frac{1}{q}}(t)\in L^{\kappa}(T), \quad \kappa = \begin{cases} \frac{pq}{p-q}, & p>q, \\ \infty, & p=q. \end{cases}
\end{equation}
where $J(t)$ is a Radon--Nikod\'ym derivative of measure 
$(N(Q)\cdot\nu)\circ\psi^{-1}$ with respect to $\mu$.
Here $N(Q)(s) = \|\mathcal{H}_{\xi(s,\cdot)}\|_{L^\gamma(U'_{\psi(s)})}$, $\gamma = \frac{\beta\alpha}{\beta-\alpha}$ if $\beta > \alpha$ and $\gamma = \infty$ if $\beta = \alpha$
\end{thm}

\begin{thm}Let $p\geq q$ and $\beta\geq\alpha$ and
let mapping $\varphi: U \to U'$ be such that
\begin{enumerate}
\item[(B1)] $\varphi(s,x) = (\psi(s), \xi(s,x))$ preserves the priority of the variables,
\item[(B2)] $\xi(s, \cdot)$ is a mapping with bounded $(\beta, \alpha)$-distortion for every $s \in S$,
\item[(B3)] $\psi:S\to T$ is injective and enjoys Luzin $\mathcal{N}$-property.
\end{enumerate}
Then $\varphi$ induces the bounded operator
$C_{\varphi}:L^p(T, W^{1,\beta}(U'_t)) \to L^q(S, W^{1,\alpha}(U_s))$, if and only if
\begin{equation}
J^{\frac{1}{q}}_{\psi^{-1}}(t)\mathcal{H}_{\xi(\psi^{-1}(t),\cdot)}(y) \in L^{\kappa, \gamma}(U').
\end{equation}
\end{thm}

\bibliographystyle{plain}
\bibliography{direct}
\end{document}